\renewcommand\H{\mathrm{H}}
\renewcommand\L{\mathrm{L}}
\newcommand\W{\mathrm{W}}
\newcommand\R{\mathbb{R}}
\newcommand\bbH{\mathbb{H}}

\newcommand\bH{\mathbf{H}}
\newcommand\bL{\mathbf{L}}
\newcommand\cC{\mathcal{C}}
\newcommand\cD{\mathcal{D}}
\newcommand\cA{\mathcal{A}}
\newcommand\cT{\mathcal{T}}
\newcommand\cP{\mathcal{P}}
\newcommand\cV{\mathcal{V}}

\newcommand{\cF}{\mathcal{F}}

\renewcommand\t{\mathtt{t}}
\newcommand\bn{\boldsymbol{n}}
\newcommand\bF{\boldsymbol{f}}
\newcommand\bu{\boldsymbol{u}}
\newcommand\bv{\boldsymbol{v}}
\newcommand\br{\boldsymbol{r}}
\newcommand\bs{\boldsymbol{s}}
\newcommand\be{\boldsymbol{e}}
\newcommand\beps{\boldsymbol{\varepsilon}}
\newcommand\bsig{\boldsymbol{\sigma}}
\newcommand\btau{\boldsymbol{\tau}}
\newcommand\bze{\boldsymbol{\zeta}}
\newcommand\bet{\boldsymbol{\eta}}
\newcommand\bgam{\boldsymbol{\gamma}}

\newcommand\kq{\mathfrak{q}}
\newcommand\kp{\mathfrak{p}}
\newcommand\bdiv{\mathop{\mathbf{div}}\nolimits}

\newcommand{\norm}[1]{\left\| #1\right\|}
\newcommand\tr{\mathop{\mathrm{tr}}\nolimits}
\newcommand{\dual}[1]{\langle #1 \rangle}
\def\<#1>{\mathinner{\langle#1\rangle}}

\newcommand{\set}[1]{\lbrace #1 \rbrace}

\newcommand{\jump}[1]{\llbracket #1 \rrbracket}
\newcommand{\mean}[1]{\{#1\}}

\documentclass[11pt]{article}
\usepackage{amsmath,amsfonts,amsthm, amssymb, latexsym, graphics, color, stmaryrd}
\usepackage{booktabs, multirow}

\DeclareMathOperator*{\esssup}{ess\,sup}
\newtheorem{lemma}{Lemma}[section]
\newtheorem{theorem}{Theorem}[section]
\newtheorem{prop}{Proposition}[section]
\newtheorem{corollary}{Corollary}[section]

\numberwithin{equation}{section}
\numberwithin{figure}{section}
\numberwithin{table}{section}

\date{}
\title{Mixed-hybrid and mixed-discontinuous Galerkin methods for linear dynamical elastic-viscoelastic composite structures
\thanks{This research was  supported by Spain's Ministry of Economy Project MTM2017-87162-P.}}

\author{ {\sc Antonio M\'arquez}\thanks{Departamento de Construcci\'on e Ingenier\'\i a de Fabricaci\'on, Universidad de Oviedo, Oviedo, Espa\~na, e-mail: {\tt amarquez@uniovi.es}.}
\quad and \quad 
{\sc Salim Meddahi}\thanks{Departamento de Matem\'aticas, Facultad de Ciencias, Universidad de Oviedo, Calvo Sotelo s/n, Oviedo, Espa\~na, e-mail: {\tt salim@uniovi.es}.}}

\begin{document}

\maketitle

\begin{abstract}
\noindent
We introduce and analyze a stress-based formulation for Zener's model in linear viscoelasticity. The method is aimed to tackle efficiently  heterogeneous materials that admit purely elastic and viscoelastic parts in their composition. We write the mixed variational formulation of the problem in terms of a class of tensorial wave equation and obtain an energy estimate that guaranties the well-posedness of the problem through a standard Galerkin procedure. We propose and analyze mixed continuous and discontinuous Galerkin space discretizations of the problem and derive optimal error bounds for each semidiscrete solution in the corresponding energy norm. 
\end{abstract}
\medskip

\noindent
\textbf{Mathematics Subject Classification.} 65N30, 65M12, 65M15, 74H15
\medskip

\noindent
\textbf{Keywords.}
mixed finite elements, elastodynamics, error estimates

\section{Introduction}  

We aim to study the propagation of waves in linear composite structures constituted by elastic and viscoelastic materials placed in juxtaposition to each other. The typical viscoelastically damped structure is composed of core layers of viscoelastic components interspersed between layers of purely elastic materials. The behaviour of this kind of engineering media under dynamic loading is employed in a wide range of applications, cf. \cite{sayyad} and the references therein. For instance, the ability of laminated viscoelastic materials to absorb vibrational energy is often exploited to damp mechanical shocks or to attenuate resonant vibrations and control noise propagation.  

Viscoelastic materials exhibit time-dependent strain effects in response to applied loads. The stress-strain time relationship (see \eqref{oldConstitutive} below) proposed by Zener \cite{zener}  gives rise to  the simplest viscoelastic model that takes into account important features such as creep/recovery and stress relaxation phenomena \cite{salencon}. The parameter $\omega\geq 0$ in \eqref{oldConstitutive} is called the characteristic relaxation time, it measures the extent to which the model includes viscoelastic effects for the solid constituent, with $\omega = 0$ corresponding to the purely elastic case. The scope of this work is to propose a new mixed variational formulation in viscoelasticity that is able to  handle composites made up of elastic and viscoelastic components. This amounts to allow the relaxation time $\omega$  to be a piecewise constant function  that may  vanish identically in parts of the domain representing the structure. 

Displacement based formulations in viscoelasticity are generally written in terms of the hereditary integral form of the constitutive equation \eqref{oldConstitutive}. The resulting problem consists in a second-order hyperbolic partial differential equation with a fading memory Volterra integral. This classical approach is widely used in engineering and it has been the subject of many mathematical \cite{Morro, gurtin} and numerical \cite{Shaw1, riv1, Janovsky, Idesman} studies. This formulation is generally solved numerically by a Galerkin finite element discretization in space, followed by a finite difference method for the time derivatives and a quadrature rule for the Volterra integral. Such a time-stepping scheme requires to keep in memory solutions of past time levels, which may cause great demands for data storage. For this reason, some studies have opted for avoiding  time convolution by incorporating further (memory) variables to the formulation of problem, see \cite{Carcione, riv2}.

We are interested here in differential formulations of the viscoelastic problem based on  dual mixed formulations of the viscoelastic problem \cite{Becache, rognes, lee, gmm-2020}. This approach provides direct and  accurate approximations of the stress tensor, which is the quantity of primary interest in many applications. Moreover, the mixed formulations in elasticity are known to be free from volumetric locking in the nearly incompressible case. The approach requires to separate the stress tensor into two independent components representing  the elastic and viscoelastic contributions. Rognes and Winther shown in \cite{rognes} that it is possible to exploit well-known stable  families of finite elements for mixed approximations in elasticity \cite{abd, ArnoldFalkWinther,CGG,GG} to solve the quasi-static Maxwell and Kelvin-Voigt models in viscoelasticity.  This strategy has  been generalized by Lee \cite{lee} for the dynamic standard linear model. In contrast to the first order formulation in time considered in \cite{Becache, lee}, the recent article   \cite{gmm-2020} proposes to directly work with the equations in second order form, which permits to stabilise  the stress tensor in the full energy norm and obtain H(div)-norm convergence error estimates for this variable. 

None of the aforementioned methods considers the case of a constitutive law that reduces to Hooke's law in parts of the viscoelastic body. We are not aware about results in the literature analysing mixed formulations for composite elastic-viscoelastic structures. The purpose of this paper is to reformulate the mixed finite element method presented in \cite{gmm-2020} in order to efficiently address such materials. We propose two different space discretizations of the problem:
\begin{itemize}
 
	\item The first option is based on a continuous Galerkin (CG) method. Unlike \cite{rognes, lee, gmm-2020} where the discrete counterpart of the elastic and viscoelastic components of the stress tensor are both sought in H(div), we release here each of these components from any continuity requirement and only ask their sum (which represents the real stress variable) to be H(div)-conforming. We use standard energy techniques and the well-known  stability properties of the Arnold-Falk-Whinter (AFW) finite element \cite{ArnoldFalkWinther} of order $k\geq 1$  to carry out the convergence analysis of this semi-discrete formulation. We obtain optimal error estimates for the stress tensor in the full H(div)-norm. 
	
	We propose a second order accurate implicit time stepping scheme for the time discretization of the problem. We point out that it is generally not straightforward to compute explicitly the basis functions of the finite element space described above. We overcome this drawback by use of a hybridization technique (in the spirit of \cite{ab}): We enforce the  continuity requirement of the normal components of the stress tensor across the interelement boundary through a Lagrange multiplier representing the trace of the velocity field on the boundaries of the mesh elements. Then, we use a static condensation procedure to recast, at each iteration step, the global linear system of equations in terms of the single hybrid variable. After solving for the Lagrange multiplier, the remaining unknowns can be recovered by local and independent calculations on each element. In this way, the computational cost required at iteration step is comparable to that of the hybridized version of the $k^{\text{th}}$-order AFW element for elasticity.

	\item We also propose  and analyze the classical symmetric interior penalty discontinuous Galerkin (DG) method \cite{DiPietroErn} for spatial discretization adapted to an H(div)-setting \cite{conThan}. In particular, we derive optimal and arbitrary order accurate error bounds in the DG-energy norm. DG schemes are known for their ability to handle complicated meshes and discontinuous data, and for being ideally suited for $hp$-adaptivity. Unfortunately, it is not possible to derive a hybridized version of this particular DG method in order to reduce its computational cost. However,  the mass matrix arising from this spatial DG discretization is block-diagonal, with block size equal to the  number of degrees of freedom per element. It can be inverted at very low computational cost. Consequently, we suggest an explicit centered finite difference scheme for the time variable in order to obtain a fully explicit time stepping method that can be  computationally  competitive if the CFL condition is not too restrictive. 

\end{itemize}
 
 
The plan of the paper is as follows. We begin by introducing in the next section the notations,  definitions, and basic results that facilitate the derivation, in Section 3, of a new mixed formulation for Zener's model of linear viscoelasticity.  In Section 4 we obtain an energy estimate and illustrate how it can be used to prove the well-posedness of the problem by means of a Galerkin procedure. In Section 5 we recall basic properties of the AFW mixed finite element and we construct a projector that plays an essential role in Sections 6 and 7, where we introduce and analyse the convergence of CG and DG semidiscrete approximations of the problem, respectively. In both cases, we provide a priori error bounds for the semidiscrete solution in the energy norms. In Section 8, we take into account the different characteristics of  the CG and DG methods to propose fully discrete schemes that are  specially tailored for each case. 

\section{Notations and preliminary results}\label{section2}
In this section we provide the notations,  definitions, and preliminary results that will be  employed all along the paper. We first denote by $\boldsymbol{I}$ the identity matrix of $\R^{d \times d}$ ($d=2,3$), and by $\mathbf{0}$ the null vector in $\R^d$ or the null tensor in $\R^{d \times d}$. In addition, the component-wise inner product of two matrices $\bsig, \,\btau \in\R^{d \times d}$ is defined by $\bsig:\btau:= \tr( \bsig^{\t} \btau)$, where $\tr\btau:=\sum_{i=1}^d\tau_{ii}$ and $\btau^{\t}:=(\tau_{ji})$ stand for the trace and the transpose of $\btau = (\tau_{ij})$, respectively. In turn, for $\bsig:\Omega\to \R^{d \times d}$ and $\bu:\Omega\to \R^d$, we set the row-wise divergence $\bdiv \bsig:\Omega \to \R^d$, the  row-wise gradient $\nabla \bu:\Omega \to \R^{d \times d}$, and the strain tensor $\beps(\bu) : \Omega \to \R^{d \times d}$ as
\[
(\bdiv \bsig)_i := \sum_j   \partial_j \sigma_{ij} \,, \quad (\nabla \bu)_{ij} := \partial_j u_i\,,
\quad\hbox{and}\quad \beps(\bu) := \frac{1}{2}\left[\nabla\bu+(\nabla\bu)^{\t}\right]\,,
\]
respectively. Next, we let $\Omega$ be a polyhedral Lipschitz bounded domain of $\R^d$ $(d=2,3)$, with boundary $\partial\Omega$. Furthermore, for $s\in \mathbb{R}$, $\norm{\cdot}_{s,\Omega}$ stands indistinctly for the norm of the Hilbertian Sobolev spaces $\H^s(\Omega)$, $\bH^s(\Omega):= [\H^s(\Omega)]^d$ or $\mathbb{H}^s(\Omega):= [\H^s(\Omega)]^{d \times d}$, with the convention $\H^0(\Omega):=\L^2(\Omega)$. In all what follows, $(\cdot, \cdot)$ stands for the inner product in $\L^2(\Omega)$, $\bL^2(\Omega):=[\L^2(\Omega)]^d$, $\mathbb{L}^2(\Omega):=[\L^2(\Omega)]^{d \times d}$, and $\mathfrak{L}^2(\Omega) := \mathbb{L}^2(\Omega) \times \mathbb{L}^2(\Omega)$. We also introduce the Hilbert space  $\mathbb{H}(\bdiv, \Omega):=\big\{\btau\in\mathbb{L}^2(\Omega):\ \bdiv\btau\in\bL^2(\Omega)\big\}$ and denote the corresponding norm $\norm{\btau}^2_{\mathbb H(\bdiv,\Omega)}:=\norm{\btau}_{0,\Omega}^2+\norm{\bdiv\btau}^2_{0,\Omega}$. Let $\bn$ be the outward unit normal vector to $\partial \Omega$. It is well-known that the normal trace operator $ [\cC^\infty(\overline\Omega)]^{d\times d}\ni \btau \to (\btau|_{\partial \Omega})\bn$ can be extended to a linear continuous mapping $(\cdot|_{\partial \Omega})\bn:\, \mathbb{H}(\bdiv, \Omega) \to \bH^{-\frac{1}{2}}(\partial \Omega)$, where $\bH^{-\frac{1}{2}}(\partial \Omega)$ is the dual of $\bH^{\frac{1}{2}}(\partial \Omega)$.

\paragraph{Sobolev spaces for time dependent problems.}
Since we will deal with a time-domain problem, besides the Sobolev spaces defined above, we need to introduce spaces of functions acting on a bounded time interval $(0,T)$ and with values in a separable Hilbert space $V$, whose norm is denoted here by $\norm{\cdot}_{V}$. In particular, for $1 \leq p\leq \infty$, $\L^p(V)$ is the space of classes of functions $f:\ (0,T)\to V$ that are B\"ochner-measurable and such that $\norm{f}_{\L^p(V)}<\infty$, with 
\[
\norm{f}^p_{\L^p(V)}:= \int_0^T\norm{f(t)}_{V}^p\, \text{d}t\quad \hbox{for $1\leq p < \infty$},
\quad\hbox{and} \quad \norm{f}_{\L^\infty(V)}:= \esssup_{[0, T]} \norm{f(t)}_V.
\]
We use the notation $\mathcal{C}^0(V)$ for the Banach space consisting of all continuous functions $f:\ [0,T]\to V$. More generally, for any $k\in \mathbb{N}$, $\mathcal{C}^k(V)$ denotes the subspace of $\mathcal{C}^0(V)$ of all functions $f$ with (strong) derivatives $\frac{\text{d}^j f}{\text{d}t^j}$ in $\mathcal{C}^0(V)$ for all $1\leq j\leq k$. In what follows, we will use indistinctly the notations $\dot{f}:= \frac{\text{d} f}{\text{d}t}$ and $\ddot{f} := \frac{\text{d}^2 f}{\text{d}t^2} $ to express the first and second derivatives with respect to the variable $t$. Furthermore,  we will use the Sobolev space
\[
\begin{array}{c}
\W^{1, p}(V):= \left\{f: \ \exists g\in \L^p( V)
\ \text{and}\ \exists f_0\in V\ \text{such that}\
 f(t) = f_0 + \int_0^t g(s)\, \text{d}s\quad \forall t\in [0,T]\right\}.
\end{array}
\]
We define the space $\W^{k, p}(V)$ recursively  for all $k\in\mathbb{N}$ and denote $\H^k(V):= \W^{k,2}(V)$.

\paragraph{The inf-sup condition and the closed range theorem.}
Given two Hilbert spaces $S$ and $Q$ and a bounded bilinear form $a:S\times Q\to\R$, we denote $\ker(a):=\set{s\in S:\ a(s,q)=0\ \forall\, q\in Q}$.
We say that $a$ satisfies the inf-sup condition for the pair $\{ S,Q\}$, whenever there exists $\kappa>0$ such that
\begin{equation}\label{0infsup}
	\sup_{0\neq s\in S}\frac{a(s,q)}{\norm{s}_{ S}} \ge \kappa \, \norm{q}_{Q} \qquad\forall\, q\in Q.
\end{equation}
We will repeatedly use the well-known fact  that (see \cite{BoffiBrezziFortinBook}) if $a$ satisfies the inf-sup condition for the pair $\{ S, Q\}$ and if $\ell\in S'$ vanishes identically on $\ker(a)$, then there exists a unique $q\in Q$ such that 
\[
a(s,q) = \ell(s) \quad \forall\, s\in S\,.
\]

Throughout the rest of the paper, given any positive expressions $X$ and $Y$ depending on the meshsize  $h$ of a triangulation, the notation $X \,\lesssim\, Y$  means that $X \,\le\, C\, Y$ with a constant $C > 0$ independent of the mesh size $h$.

\section{A mixed variational formulation of the Zener model}\label{section3}
We aim to study  the dynamical equation of motion 
\[
\rho\ddot{\bu}-\bdiv\bsig =\bF \quad \text{in $\Omega\times (0, T]$},
\]
for a viscoelastic body represented by a polyhedral Lipschitz domain  $\Omega\subset \mathbb R^d$ ($d=2,3$). The constant $T>0$ stands for the time interval endpoint,  $\bu:\Omega\times[0, T] \to \R^d$ is the displacement field, $\bsig:\Omega\times[0,T]\to \R^{d\times d}$ is the stress tensor and $\bF:\Omega\times[0, T] \to \R^d$ represents the body force. The linearized strain tensor $\beps(\bu)$ is assumed to determine stress through Zener's constitutive law for viscoelasticity (see \cite{salencon}):
\begin{equation}\label{oldConstitutive}
	\bsig +  \omega \dot{\bsig} = \cC \beps(\bu) + \omega \cD \beps(\dot{\bu}) \quad \text{in $\Omega\times (0, T]$},
\end{equation}
where $\cC$ and $\cD$ are two  symmetric and positive definite tensors of order 4.  To guaranty that the system is dissipative we assume that $\cD - \cC$ is also positive definite. We suppose that the mass density $\rho$ and the relaxation time $\omega$ are piecewise constant functions. More specifically, we assume  that there exists a  disjoint partition of $\bar \Omega$ into polygonal/polyhedral subdomains $\big\{\bar\Omega_j,\ j= 1,\ldots,J\big\}$  such that $\rho|_{\Omega_j}:= \rho_j>0$ and $\omega|_{\Omega_j}:= \omega_j\geq 0$ for  $j=1,\ldots,J$. We point out that, in the regions $\Omega_j$ where the piecewise constant function $\omega$ is zero,  the constitutive law \eqref{oldConstitutive} becomes the familiar Hooke's Law. Hence, it is natural to  introduce the set of indices $\mathcal I_E := \big\{j\in \set{1,\ldots, J}:\ \omega_j = 0\big\}$ and $\mathcal I_V := \big\{j\in \set{1,\ldots, J}:\ \omega_j \neq 0\big\}$ and to split $\Omega$ into  a part $\Omega_E := \cup_{j\in \mathcal I_E} \Omega_j$ displaying a purely elastic behaviour and a part  $\Omega_V:= \cup_{j\in \mathcal I_V} \Omega_j$ exhibiting viscoelastic properties. In the sequel, we will need  the piecewise constant function $\tilde \omega^{-1}$ defined by 
\[
\tilde \omega^{-1}|_{\Omega_j} :=\begin{cases}
	\omega_j^{-1} & \text{if $j\in \mathcal I_V$},
	\\
	0 & \text{if $j\in \mathcal I_E$},
\end{cases} \quad j= 1,\ldots, J.
\]

For the sake of simplicity in exposition, we restrict our analysis to the case of a prescribed displacement field on $\Gamma:=\partial \Omega$. Namely, we impose the Dirichlet boundary condition $\bu = \boldsymbol{g}$ on $\Gamma \times (0, T]$, with $\boldsymbol{g}:\, \Gamma \times (0, T]\to \mathbb R^d$ given. Finally, we assume the initial conditions:
\begin{equation}\label{init1}
	\bu(0) = \bu_0 \quad\text{in $\Omega$}, \quad 
 \dot{\bu}(0) = \bu_1 \quad\text{in $\Omega$},   \quad \text{and} \quad 
 \bsig(0) = \bsig_0 \quad\text{in $\Omega_V$}.
\end{equation}

Our aim is to impose the stress tensor $\bsig$ as a primary unknown. To this end, we  decompose this variable into a purely elastic component $\bgam:= \cC \beps(\bu)$ and a viscoelastic component $\omega \bze:= \bsig - \bgam$, which allows us to transform the constitutive law \eqref{oldConstitutive} into  
\[
\omega^2 \dot \bze + \omega \bze = \omega (\cD - \cC) \beps(\dot \bu).
\]
Hence, if we adopt the notations $\cA := \cC^{-1}$ and $\cV:= (\cD - \cC)^{-1}$, our model problem can be written as follows in terms of $\bu$, $\bgam$, and $\bze$:
 \begin{align}\label{split123}
 \begin{split}
 \rho\ddot{\bu}-\bdiv(\bgam + \omega \bze) &=\bF  \quad\text{in $\Omega\times (0, T]$}, 
 \\[1ex]
 (\bgam + \omega \bze) &= (\bgam + \omega \bze)^{\t} \quad \text{in $\Omega\times (0, T]$},
 \\[1ex]
 \cA \ddot{\bgam} &= \beps(\ddot{\bu})  \quad \text{in $\Omega\times (0, T]$}, 
 \\[1ex]
 \omega^2 \cV \ddot{\bze} +  \omega\cV \dot{\bze} &= \omega\beps(\ddot{\bu}) \quad \text{in $\Omega\times (0, T]$},
 \\[1ex]
 \bu &= \boldsymbol{g} \quad  \text{on $\Gamma\times (0, T]$}.
\end{split}
\end{align}
The main unknown of our formulation is then a pair of tensors $\kp:=(\bgam, \bze)\in \mathfrak{L}_V^2(\Omega):= \mathbb L^2(\Omega)  \times \mathbb L^2(\Omega_V)$, where
\[
\mathbb L^2(\Omega_V) := \set{\btau \in \mathbb L^2(\Omega):\ \btau|_{\Omega_E} = \mathbf 0 }.
\]
  We introduce the linear and bounded operators 
\begin{align*}
  \jmath_\omega:\,  \mathfrak{L}_V^2(\Omega) \, &\longrightarrow \,  \mathfrak{L}_V^2(\Omega) 
  &
   \jmath_\omega^+:\,  \mathfrak{L}_V^2(\Omega) \, &\longrightarrow \,  \mathbb{L}^2(\Omega) 
  \\[1ex]
  \kq= (\bet, \btau) &\longrightarrow \, \jmath_\omega \kq := (\bet, \omega\btau)  
  &
  \kq= (\bet, \btau) &\longrightarrow \, \jmath_\omega^+ \kq := \bet + \omega\btau
\end{align*}
and consider the space  
\[
 \mathfrak{S}:= 
\big\{ \kq \in \mathfrak{L}_V^2(\Omega):\ \jmath_\omega^+\kq \in \mathbb H(\bdiv, \Omega)  \big\},
\]
endowed with the Hilbertian norm  
\[
\norm{\kq}^2_{\mathfrak{S}}:= \norm{\jmath_\omega \kq}_{\mathfrak{L}_V^2(\Omega)}^2+ \norm{\bdiv \jmath_\omega^+\kq}^2_{0,\Omega} = \norm{\bet}^2_{0,\Omega} + \norm{\omega\btau}^2_{0,\Omega_V}  + \norm{\bdiv (\bet + \omega \btau) }^2_{0,\Omega},\quad \kq = (\bet, \btau). 
\] 
To take into account the symmetry of the stress tensor (second equation of \eqref{split123}), we introduce the space $\mathbb{L}^2_{\text{sym}}(\Omega):= \big\{\btau \in \mathbb{L}^2(\Omega):\ \btau = \btau^{\t}\big\}$ and let  $\mathfrak L_{\text{sym}}^2(\Omega):= \set{\kq \in  \mathfrak L_V^2(\Omega):\, \jmath_\omega^+ \kq \in \mathbb{L}^2_{\text{sym}}(\Omega)}$. We will show that  $\mathfrak{S}_{\text{sym}} := \mathfrak{S}\cap \mathfrak L_{\text{sym}}^2(\Omega)$ is the adequate energy space for problem \eqref{split123}. We  consider an arbitrary $\kq = (\bet, \btau)\in \mathfrak S_{\text{sym}}$, test the third and fourth rows of \eqref{split123} with $\bet$ and $\btau$ and add the resulting equations  to get
\begin{equation}\label{const+}
(\cA \ddot{\bgam},\bet) + ( \cV (\omega\ddot{\bze} + \dot{\bze}) ,\omega \btau)  = (\beps(\ddot{\bu}), \bet + \omega\btau) = \big(\nabla \ddot\bu , \bet + \omega\btau\big),
\end{equation} 
where the last identity follows from the fact that the tensor $\bet + \omega\btau$ is symmetric. 
Next, we integrate by parts in the right hand-side of \eqref{const+} and take into account the boundary condition on $\Gamma\times (0, T]$ to obtain
\begin{equation}\label{var0}
(\cA \ddot{\bgam},\bet) + ( \cV (\omega\ddot{\bze} + \dot{\bze}) ,\omega \btau)  =   
- \big(\ddot{\bu}, \bdiv(\bet + \omega\btau)\big) + \big<\ddot{\boldsymbol{g}}, (\bet + \omega\btau)\bn \big>_\Gamma, 
\end{equation}
where $\bn$ represents the exterior unit normal vector on $\Gamma$ and $\dual{\cdot, \cdot}_\Gamma$ holds for the duality pairing between $\bH^{\frac{1}{2}}(\Gamma)$ and $\bH^{-\frac{1}{2}}(\Gamma)$.  
Substituting back $\ddot{\bu} =  \rho^{-1}\big(\bF + \bdiv(\bgam + \omega \bze)\big)$ into \eqref{var0} yields  
\begin{equation}\label{var1}
A\big( \jmath_\omega \ddot\kp + \pi_2 \dot\kp, \jmath_\omega\kq \big)  
+ \big(\bdiv\jmath_\omega^+\kp , \bdiv\jmath_\omega^+\kq \big)_\rho  = - \big(\bF , \bdiv\jmath_\omega^+\kq \big)_\rho + \big<\ddot{\boldsymbol{g}}, \jmath_\omega^+\kq\bn \big>_\Gamma,\quad \forall \kq =(\bet, \btau )\in \mathfrak S_{\text{sym}},
\end{equation}
where  $(\bu,\bv)_\rho := (\tfrac{1}{\rho}\bu, \bv)$ for all $\bu$, $\bv$ in $\bL^2(\Omega)$,  
\[
A\big( \kp,\kq\big) := (\cA\bgam, \bet) + (\cV \bze, \btau), \quad  \kp=(\bgam, \bze),\, \kq=(\bet, \btau) \in \mathfrak{L}^2(\Omega),
\]
and  where $\pi_2$ stands for the $\mathfrak{L}^2(\Omega)$-orthogonal projection  onto $\{\mathbf 0\} \times \mathbb L^2(\Omega)$, i.e., $\pi_2\kq:= (\mathbf 0, \btau)$ for all $\kq=(\bet, \btau)\in \mathfrak{L}^2(\Omega)$. It is important to notice that, as a consequence of our hypotheses on $\cC$ and $\cD$, the bilinear form $A$ is symmetric, bounded and coercive, i.e., there exist positive constants $M$ and $\alpha$, depending only on $\cC$ and $\cD$, such that
\begin{equation}\label{contA}
\big| A( \kp,\kq) \big| \leq M \| \kp \|_{0,\Omega} 
\|\kq \|_{0,\Omega} \qquad \forall \, \kp, \kq \in \mathfrak{L}^2(\Omega),
\end{equation}
\begin{equation}\label{ellipA}
 A( \kq,\kq) \geq \alpha  
\|\kq \|^2_{0,\Omega} \qquad \forall  \,\kq \in \mathfrak{L}^2(\Omega).
\end{equation}

Next, we need to introduce function spaces that are used in the study of evolution problems of second order in time with energy methods, cf. \cite[Chapter XVIII]{Lions}. We begin with the following technical result.

\begin{lemma}\label{dense}
	The continuous embedding $\mathfrak{S}_{\emph{sym}}   \hookrightarrow \mathfrak{L}_{\emph{sym}}^2(\Omega)$ is dense.
\end{lemma}
\begin{proof}
The main argument of the proof is the well-known density in $\L^2(\Omega)$ of the space  of indefinitely differentiable functions with compact support in $\Omega$. Given an arbitrary $\kq=(\bet, \btau)\in \mathfrak{L}_V^2(\Omega)$, there exist sequences $\set{\bet_n}_n$  and $\set{\btau_n}_n$ of smooth tensors converging  in the $\mathbb L^2(\Omega)$-norm to $\bet$ and $\omega\btau$, respectively. The sequence $\set{\kq_n}_n:= \set{(\bet_n, \tilde \omega^{-1}\btau_n)}_n$ is a subset of $\mathfrak S$ since, by construction, $\jmath_\omega^+ \kq_n = \bet_n+\btau_n\in \mathbb H(\bdiv, \Omega)$ and the convergence of $\{\kq_n\}_n$ to $\kq$ in $\mathfrak{L}^2(\Omega)$, proves that $\mathfrak{S}$ is a dense subset of   $\mathfrak{L}_V^2(\Omega)$. 

Now, if $\kq=(\bet, \btau)\in \mathfrak{L}_{\text{sym}}^2(\Omega)$, we still have that  $\set{\kq_n}_n:= \set{(\bet_n, \tilde \omega^{-1}\btau_n)}_n\subset \mathfrak S$ converges to $\kq$ in $\mathfrak L^2(\Omega)$.  It follows that the sequence defined by $\set{\widehat{\kq}_n}_n:=\set{ (\widehat{\bet}_n, \tilde \omega^{-1}\btau_n)}_n$, with  $\widehat{\bet}_n := \bet_n - \dfrac{\jmath_\omega^+\kq_n - (\jmath_\omega^+\kq_n)^{\t}}{2}$, belongs to $\mathfrak{S}_{\text{sym}}$ because the symmetric tensor $\jmath_\omega^+ \widehat{\kq}_n = \dfrac{\jmath_\omega^+\kq_n + (\jmath_\omega^+\kq_n)^{\t}}{2}\in \mathbb{H}(\bdiv, \Omega)$ converges to $(\bet - \dfrac{\jmath_\omega^+\kq - (\jmath_\omega^+\kq)^{\t}}{2}, \btau) =\kq$ in $\mathfrak L^2(\Omega)$, which proves the result. 
\end{proof}

Thanks to Lemma~\ref{dense}, we can identify the Hilbert space $\mathfrak L^2_{\text{sym}}(\Omega)$ with its dual and consider 
 the sequence $\mathfrak{S}_{\text{sym} } \hookrightarrow \mathfrak L^2_{\text{sym}}(\Omega) \hookrightarrow \mathfrak{S}_{\text{sym} }'$ of continuous and dense inclusions, where $\mathfrak{S}_{\text{sym} }'$ stands for the dual of $\mathfrak{S}_{\text{sym} }$. Under these conditions, it can be shown (cf. \cite{Lions}) that  if $\kp\in \L^2(\mathfrak{S}_{\text{sym} })$ and $\dot\kp\in \L^2(\mathfrak{S}_{\text{sym} }')$ then 
\[
\frac{\text{d}}{\text{d}t} (\kp(t), q) = \dual{\dot \kp(t), \kq}\quad \forall \kq \in  \mathfrak{S}_{\text{sym} } 
\]
holds in the sense of distributions on $(0, T)$, where $\dual{\cdot, \cdot}$ represents the duality brackets between $\mathfrak{S}_{\text{sym} }'$ and $\mathfrak{S}_{\text{sym} }$.  Keeping this fact in mind, we introduce the Hilbert space  
\[
W(\mathfrak{S}_{\text{sym} }, \mathfrak{S}_{\text{sym} }') := \set{\kq\in \L^2(\mathfrak{S}_{\text{sym} }):\  \dot\kq\in V(\mathfrak{S}_{\text{sym} }, \mathfrak{S}_{\text{sym} }') },
\]
where
\[
V(\mathfrak{S}_{\text{sym} }, \mathfrak{S}_{\text{sym} }') :=
\set{\kq\in \L^2(\mathfrak{S}_{\text{sym} }):\  \dot\kq\in \L^2(\mathfrak{S}_{\text{sym} }') },
\]
and consider the following variational formulation of \eqref{split123}:  Given $\bF\in \L^{2}(\bL^2(\Omega))$ and $\boldsymbol{g}\in \H^{2}(\bH^{1/2}(\Gamma))$, we look for $\kp\in W(\mathfrak{S}_{\text{sym} }, \mathfrak{S}_{\text{sym} }')$ satisfying 
\begin{align}\label{varFormR1-varFormR2}
\begin{split}
\frac{\text{d}}{\text{d}t}   A\big( \jmath_\omega \dot\kp + \pi_2 \kp, \jmath_\omega\kq \big)  + \big(\bdiv\jmath_\omega^+\kp , \bdiv\jmath_\omega^+\kq \big)_\rho
&= - \big(\bF , \bdiv\jmath_\omega^+\kq \big)_\rho + \big<\ddot{\boldsymbol{g}}, \jmath_\omega^+\kq\bn \big>_\Gamma,\quad \forall \kq \in \mathfrak{S}_{\text{sym} },
\\[1ex]
\kp(0) &= \kp_0, \qquad \dot\kp(0) = \kp_1,
 \end{split}
\end{align} 
where $\kp_0 = (\bgam_0, \bze_0)\in \mathfrak{S}_{\text{sym}}$ and $\kp_1= (\bgam_1, \bze_1) \in \mathfrak{L}_{\text{sym}}^2(\Omega)$ with 
\begin{align}\label{initcompatible}
	\begin{split}
		\bgam_0 &:= \cC \beps(\bu_0),\quad   \bze_0 = \tilde\omega^{-1} \big( \bsig_0 - \bgam_0 \big),
\\[1ex]
\bgam_1 &:= \cC \beps(\bu_1), \quad\bze_1 := \tilde\omega^{-1} \big( \cD \beps(\bu_1) - \bgam_1 - \bze_0\big).
	\end{split}
\end{align}

We notice that the initial conditions are meaningful because of the embeddings  $\H^1(\mathfrak{S}_{\text{sym}}) \hookrightarrow \cC^0(\mathfrak{S}_{\text{sym}})$ and $V(0, T, \mathfrak{S}_{\text{sym} }, \mathfrak{S}_{\text{sym} }')\hookrightarrow  \cC^0( \mathfrak{L}^2_{\text{sym} }(\Omega))$, see \cite[Chapter XVIII, Section 1, Theorem 1]{Lions}. 

We will show that in the next section that problem \eqref{varFormR1-varFormR2} is  well-posed if some regularity assumptions on  $t\mapsto \bF(t)$ and $t\mapsto\boldsymbol{g}(t)$ are fulfilled.

\section{Existence and uniqueness}

We aim  to obtain formal energy estimates for \eqref{varFormR1-varFormR2} in terms of the energy functional $\mathcal{E}:\, W(\mathfrak{S}_{\text{sym} }, \mathfrak{S}_{\text{sym} }') \to \cC^0([0,T])$  defined by
\begin{equation}\label{defE}
\mathcal{E}\big(\kq\big)(t):= \frac{1}{2} A\big(\jmath_\omega\dot\kq(t), \jmath_\omega\dot\kq(t)\big) + \frac{1}{2}  
\big( \bdiv \jmath_\omega^+\kq(t), \bdiv \jmath_\omega^+\kq(t) \big)_\rho.
\end{equation}

\begin{lemma}
	Assume that $\bF \in \H^1(\bL^2(\Omega))$ and $\boldsymbol{g}\in \H^3(\bH^{1/2}(\Gamma))$. Then,  if $\kp$ is a solution of  \eqref{varFormR1-varFormR2}, it satisfies
\begin{equation}\label{AprioriBound_h}
\max_{t\in [0,T]}\mathcal{E}(\kp)^{1/2}(t)  
\lesssim \norm{\bF}_{\H^1(\bL^2(\Omega))} +   \norm{ \boldsymbol{g} }_{\H^3(\bH^{1/2}(\Gamma))} + \norm{\kp_0}_{\mathfrak{S}}  + \norm{\jmath_\omega \kp_1}_{\mathfrak{L}_V^2(\Omega)}.
\end{equation}
\end{lemma}
\begin{proof}
We take $\kq = \dot\kp = (\dot\bgam, \dot\bze)\in \mathfrak S_{\text{sym}}$ in the first equation of \eqref{varFormR1-varFormR2} and integrate the resulting identity over $(0, t)$ to obtain    
\[ 
\mathcal{E}\big(\kp\big)(t) - \mathcal{E}\big(\kp\big)(0) + \int_0^t (\omega\cV \dot\bze(s), \dot\bze(s) ) \, \text{d}s =   - \int_0^t \big( \bF(s),  \bdiv \dot\jmath_\omega^+\dot \kp(s) \big)_\rho \, \text{d}s + \int_0^t   \big<\ddot{\boldsymbol{g}}(s), \dot\jmath_\omega^+\dot \kp(s)\bn \big>_\Gamma\, \text{d}s.
\]
Next, we notice that the last term on the left-hand side is non-negative and we integrate by parts on the right-hand side  to find
\begin{align*}
	\mathcal{E}\big(\kp\big)(t)  - \mathcal{E}\big(\kp\big)(0)&
	\leq \int_0^t \big( \dot{\bF}(s), \bdiv\jmath_\omega^+\kp(s)\big)_\rho \, \text{d}s -  \big( \bF(t), \bdiv \jmath_\omega^+\kp(t)\big)_\rho + \big( \bF(0), \bdiv \jmath_\omega^+\kp_{0}\big)_\rho 
\\[1ex]
&
\quad - \int_0^t \big<\frac{\text{d}^3\boldsymbol{g} }{\text{d}t}(s),  \jmath_\omega^+ \kp(s)\bn \big>_\Gamma \, \text{d}s + \big<\ddot{\boldsymbol{g}}(t),  \jmath_\omega^+ \kp(t)\bn \big>_\Gamma - \big<\ddot{\boldsymbol{g}}(0),  \jmath_\omega^+ \kp_0\bn \big>_\Gamma
\end{align*}
By virtue of the Cauchy-Schwarz inequality, the Sobolev embeddings $\H^1(\bL^2(\Omega)) \hookrightarrow \cC^0(\bL^2(\Omega))$ and $\H^1(\bH^{1/2}(\Gamma)) \hookrightarrow \cC^0(\bH^{1/2}(\Gamma))$ (see \cite{Lions}), and the normal trace theorem we have that 
\begin{align}\label{estimN0}
\mathcal{E}\big(\kp\big)(t)   \lesssim 
 \|\bF\|_{\H^1(\bL^2(\Omega))} \max_{t\in [0,T]}\mathcal{E}\big(\kp\big)^{1/2}(t)  + \norm{ \boldsymbol{g} }_{\H^3(\bH^{1/2}(\Gamma))}  \max_{t\in [0,T]} \norm{\jmath_\omega^+ \kp(t)}_{\mathbb H(\bdiv, \Omega)} + \mathcal{E}\big(\kp\big)(0).
\end{align}
Moreover,  thanks to the triangle inequality, the identity $\kp(t) = \int_0^t \dot\kp(s)\, \text{d}t + \kp_0$ and \eqref{ellipA} we have that 
\begin{align*}
	\norm{\jmath_\omega^+ \kp(t)}_{\mathbb H(\bdiv, \Omega)} &\lesssim \norm{\jmath_\omega \kp}_{\mathfrak{L}_V^2(\Omega)} + \norm{\bdiv \jmath_\omega^+ \kp}_{0,\Omega} \\
	& \lesssim \norm{\jmath_\omega \dot \kp}_{\mathfrak{L}_V^2(\Omega)} + \norm{\bdiv \jmath_\omega^+ \kp}_{\mathfrak{L}_V^2(\Omega)} + \norm{\kp_0}_{\mathfrak{L}_V^2(\Omega)}
	\lesssim \mathcal{E}\big(\kp\big)^{1/2}(t) + \norm{\kp_0}_{\mathfrak{L}_V^2(\Omega)},
\end{align*}
and it is clear from the definition of $\mathcal{E}$ that $\mathcal{E}\big(\kp\big)(0) \lesssim \norm{\kp_{0}}^2_{\mathfrak S} + \norm{\jmath_\omega\kp_{1}}^2_{\mathfrak{L}_V^2(\Omega)}$. Using the last estimates in \eqref{estimN0} permits us to deduce, after straightforward manipulations,  that
\begin{equation}\label{interm}
\max_{t\in [0,T]}\mathcal{E}\big(\kp\big)^{1/2}(t) \lesssim \|f\|_{\H^1(\bL^2(\Omega))} +  \norm{ \boldsymbol{g} }_{\H^3(\bH^{1/2}(\Gamma))} + \norm{\kp_{0}}_{\mathfrak{S}}+ \norm{\jmath_\omega\kp_{1}}_{\mathfrak{L}_V^2(\Omega)},
\end{equation}
and the results follows.
\end{proof}

\begin{theorem}\label{theorem-R1-R2}
Assume that $\bF \in \H^1(\bL^2(\Omega))$ and $\boldsymbol{g}\in \H^3(\bH^{1/2}(\Gamma))$. Then,  problem \eqref{varFormR1-varFormR2} admits a unique solution. Moreover, there holds
\begin{align}\label{AprioriBound}
\begin{split}
\max_{t\in [0,T]}\|\kp(t)\|_{\mathfrak{S}} &+ \max_{t\in [0,T]}\|\jmath_\omega\dot \kp(t)\|_{\mathfrak{L}_V^2(\Omega)} \lesssim  \norm{\bF}_{\H^1(\bL^2(\Omega))} + \norm{ \boldsymbol{g} }_{\H^3(\bH^{1/2}(\Gamma))} + \|\kp_0\|_{\mathfrak{S}} + \|\jmath_\omega \kp_1\|_{\mathfrak{L}_V^2(\Omega)}.  
\end{split}
\end{align}
\end{theorem}
\begin{proof}
We deduce from the definition of $\mathcal E$ and the coerciveness of $A$ (cf. \eqref{ellipA}) that 
\begin{equation}\label{eq-extra-1}
\|\jmath_\omega\dot\kq(t)\|^2_{\mathfrak{L}_V^2(\Omega)} +
\|\bdiv\jmath_\omega^+\kq(t)\|^2_{0,\Omega} \lesssim
\mathcal{E}\big(\kq\big)(t),\quad \forall \kq \in W(\mathfrak{S}_{\text{sym} }, \mathfrak{S}_{\text{sym} }').
\end{equation}
Hence, it follows from $\kp (t) = \int_0^t \dot \kp(s)\, \text{d}t + \kp_0$ and  \eqref{AprioriBound_h} that the formal a priori estimate  \eqref{AprioriBound} is satisfied. Now, the fact that $\mathfrak S_{\text{sym}}$ is a separable Hilbert space enables us to carry out a Galerkin finite dimensional space reduction, use  classical weak compactness results and obtain a solution of problem \eqref{varFormR1-varFormR2} through a limiting process. The proof of uniqueness is also standard. For the seek of brevity, we will not give here further details on the application of this classical procedure to problem \eqref{varFormR1-varFormR2}. Instead,   we refer to \cite[Chapter XVIII, Section 5]{Lions} for a detailed presentation on the application of the Galerkin method to an  abstract evolution problem of second order in time, in which our particular case fits entirely.  See also \cite{ggm-JSC-2017, gmm-2020} for similar strategies applied to mixed formulations in  elastodynamics and viscoelasticity, respectively. 
\end{proof}

The symmetry constraint integrated in $\mathfrak S_{\text{sym}}$ is difficult to handle from the numerical point of view. For this reason, we will relax this restriction by imposing it weakly  through the variational equation, 
\[
\big(\bs, \jmath_\omega^+ \kp(t)\big) = 0\quad \forall \bs \in \mathbb Q := \big\{\btau \in \mathbb{L}^2(\Omega):\ \btau = -\btau^{\t}\big\}.
\] 
A Lagrange multiplier $\br \in \mathbb{Q}$ will then make an appearance as a further variable in our formulation. To prove its existence we first  notice that due to the  embedding $\mathbb H(\bdiv, \Omega)\times \{\mathbf 0\} \hookrightarrow \mathfrak{S}$, the inf-sup condition satisfied (cf. \cite{abd}) by the bilinear form $\big(\btau, (\bs, \bv) \big)\mapsto (\bs,\btau) + (\bv,\bdiv \tau)$ for the pair $\{\mathbb H(\bdiv, \Omega),   \mathbb{Q} \times\bL^2(\Omega)\}$ (see \eqref{0infsup}) implies immediately that  there exists $\beta>0$ such that 
\begin{equation}\label{InfSupC}
\sup_{\kq \in \mathfrak S} 
\frac{(\bs,\jmath_\omega^+\kq) \,+\, \big(\bv,\bdiv \jmath_\omega^+\kq\big)}{\norm{\kq}_{\mathfrak S}}
\,\ge\, \beta \Big\{ \norm{\bs}_{0,\Omega} + \norm{\bv}_{0,\Omega} \Big\},\quad \forall (\bs,\bv) \in \mathbb{Q} \times\bL^2(\Omega).
\end{equation} 
Now, integrating \eqref{varFormR1-varFormR2} with respect to time we deduce that the functional 
 \begin{align*}
\mathcal{G}(t)\big(\kq\big)  := A\big(\jmath_\omega\dot \kp(t) &+  \pi_2\kp(t), \jmath_\omega\kq\big ) - A\big( \jmath_\omega\kp_{1} +  \pi_2 \kp_{0} , \jmath_\omega\kq\big)
 \\[1ex]
 & 
 + \int_0^t \big(\bF(s) + \bdiv \jmath_\omega^+\kp(s),\,  \bdiv \jmath_\omega^+\kq \big)_\rho \, \text{d}s - \int_0^t \big<\ddot{\boldsymbol{g}}(s), \jmath_\omega^+\kq\bn \big>_\Gamma \, \text{d}s.
\end{align*}
vanishes identically on the kernel $\mathfrak{S}_{\text{sym}}$ of the bilinear form $\mathfrak{S} \times \mathbb{Q} \ni \big(\kq, \br\big) \mapsto (\br,\jmath_\omega^+\kq)$. Therefore, the inf-sup condition \eqref{InfSupC} implies the existence of  a unique $\br\in \cC^0(\mathbb{Q})$ such that 
\begin{equation}\label{bRh}
 ( \br(t), \jmath_\omega^+\kq) = -\mathcal{G}(t)\big(\kq\big), \quad  
\forall\, t\in [0,T]\,, \quad \forall \, \kq \in \mathfrak{S}.
\end{equation}
Moreover, evaluating \eqref{bRh}  at $t=0$  and using again the  discrete inf-sup condition \eqref{InfSupC} we deduce the initial condition $\br(0) = \mathbf 0$. 
Differentiating now \eqref{bRh} in the sense of distributions on $(0, T)$ we conclude that  the pair $(\kp,\br) \in W(\mathfrak{S}_{\text{sym} }, \mathfrak{S}_{\text{sym} }')\times  \cC^0(\mathbb Q)$ is the unique solution of 
\begin{align}\label{varFormR1Extended}
\begin{split}
\frac{\text{d}}{\text{d}t}  \Big\{ A\big( \jmath_\omega \dot\kp + \pi_2 \kp, \jmath_\omega\kq \big)  +  (\br, \jmath_\omega^+\kq)\Big\}+ \big(\bdiv\jmath_\omega^+\kp , \bdiv\jmath_\omega^+\kq \big)_\rho
&= - \big(\bF , \bdiv\jmath_\omega^+\kq \big)_\rho + \big<\ddot{\boldsymbol{g}}, \jmath_\omega^+\kq\bn \big>_\Gamma,\quad \forall \kq \in \mathfrak{S}
\\[1ex]
 \big(\bs, \jmath_\omega^+ \kp(t)\big) &= 0, \quad \forall \bs\in \mathbb{Q},
 \end{split}
\end{align} 
that satisfies the initial conditions
\begin{align}\label{initial-R1-R2extended}
\begin{split}
\kp(0) &= \kp_0 \in \mathfrak{S}_{\text{sym}},
\quad  \dot\kp(0) = \kp_1 \in \mathfrak{L}_{\text{sym}}^2(\Omega), \quad \text{and}\quad \br(0) = \mathbf 0. 
\end{split}
\end{align} 

Finally, it follows from \eqref{InfSupC}, \eqref{bRh}, the Cauchy-Schwarz inequality and \eqref{AprioriBound} that, for all $t\in [0, T]$, 
\begin{align}\label{comb11C}
\begin{split}
\beta\,\norm{\br(t)}_{0,\Omega} &\leq \sup_{\kq\in \mathfrak{S}} \dfrac{ (\br(t),\jmath_\omega^+\kq)}{\norm{\kq}_{\mathfrak S}} = \sup_{\kq\in \mathfrak{S}} \frac{\mathcal{G}(t)\big(\kq\big)} {\norm{\kq}_{\mathfrak S}} \lesssim  
 \norm{\bF}_{\H^{1}(\bL^2(\Omega))} + \|\kp_{0}\|_{\mathfrak S} + \norm{ \boldsymbol{g} }_{\H^3(\bH^{1/2}(\Gamma))} + \norm{\jmath_\omega\kp_{1}}_{\mathfrak{L}_V^2(\Omega)}. 
\end{split}
\end{align}

The Lagrange multiplier $\br$ is usually known as the rotation. We can relate it to the velocity field as follows.
\begin{prop}
	If $\ddot \kp \in \L^2(\mathfrak L^2(\Omega))$ and $\dot\br\in \L^2(\mathbb Q)$, then $\ddot\bu\in \L^2(\bH^1(\Omega))$ and  
	\begin{equation}\label{rmeans}
		\br = \frac{1}{2}\big\{\nabla\dot\bu-(\nabla\dot\bu)^{\t}\big\} - \frac{1}{2}\big\{\nabla\bu_1-(\nabla\bu_1)^{\t}\big\}.
	\end{equation}
	\end{prop}
	\begin{proof}
		Using the identity $ \ddot \bu = \rho^{-1} \big( \bF + \bdiv \jmath_\omega^+ \kp\big)$ in the first equation of \eqref{varFormR1Extended} we get 
\[
 A\big( \jmath_\omega \ddot\kp + \pi_2 \dot\kp, \jmath_\omega\kq \big)  +  (\dot\br, \jmath_\omega^+\kq)+ \big(\ddot\bu , \bdiv\jmath_\omega^+\kq \big) = \big<\ddot{\bu}, \jmath_\omega^+\kq\bn \big>_\Gamma.
\]
By virtue of \eqref{const+}, $\beps(\ddot \bu)\in \L^2(\mathbb L^2(\Omega))$ and 
\[
  (\dot\br + \beps(\ddot \bu), \jmath_\omega^+\kq) = - \big(\ddot\bu ,\bdiv \jmath_\omega^+\kq \big) + \big<\ddot{\bu}, \jmath_\omega^+\kq\bn \big>_\Gamma\qquad \forall \kq \in \mathfrak S. 
\]
Choosing $\kq = (\bet, \mathbf 0)$, with $\bet\in [\cC^\infty(\Omega)]^{d\times d}$ supported in $\Omega$ we deduce that $\nabla \ddot\bu = \dot\br + \beps(\ddot \bu)$ and the result follows.
	\end{proof}

In contract to \eqref{varFormR1-varFormR2}, the variational formulation \eqref{varFormR1Extended}  imposes weakly the symmetry restriction on the stress tensor. This  well established strategy gives rise to mixed finite elements that are easy to implement and that require fewer degrees of freedom in comparison to the methods imposing strongly the symmetry at the discrete level. 

\section{Finite element spaces and auxiliary results}\label{section4}

We consider  shape regular meshes $\mathcal{T}_h$ that subdivide the domain $\bar \Omega$ into  triangles/tetrahedra $K$ of diameter $h_K$. The parameter $h:= \max_{K\in \cT_h} \{h_K\}$ represents the mesh size of $\cT_h$.  In what follows, we assume that $\mathcal{T}_h$ is compatible with the partition $\bar\Omega = \cup_{j= 1}^J \bar{\Omega}_j$, i.e., 
\[
\bigcup_{\substack{K\in \cT_h\\K\subset \Omega_j}} K  = \bar{\Omega}_j, \quad \forall j=1,\cdots, J.
\]

For any $s\geq 0$, we consider the broken Sobolev space 
\[
 \H^s(\cT_h):=
 \set{\bv \in \L^2(\Omega): \quad \bv|_K\in \H^s(K)\quad \forall K\in \cT_h},
\]
and define similarly the vectorial and tensorial versions $\mathbf H^s(\cT_h)$ and $\mathbb H^s(\cT_h)$, respectively. For each $\bv:=\set{\bv_K}\in \mathbf H^s(\cT_h)$ and $\btau:= \set{\btau_K}\in \mathbb H^s(\cT_h)$ the components $\bv_K$ and $\btau_K$  represent the restrictions $\bv|_K$ and $\btau|_K$. When no confusion arises, the restrictions of these functions will be written without any subscript.

Hereafter, given an integer $m\geq 0$ and a domain $D\subset \mathbb{R}^d$, $\cP_m(D)$ denotes the space of polynomials of degree at most $m$ on $D$. We introduce the space   
\[
 \cP_m(\cT_h) := \bigoplus_{K\in \cT_h} \cP_m(K) =\set{ v\in \L^2(\Omega): \ v|_K \in \cP_m(K),\ \forall K\in \cT_h }
 \]
 of piecewise polynomial functions relatively to $\cT_h$ and let 
 \[ \quad 
 \cP^V_m(\cT_h):=\set{v\in \cP_m(\cT_h):\ v|_{\Omega_E}=0}.
 \]
 For $k\geq 1$, we consider the finite element spaces
\begin{equation*}
	\mathbb{W}_h := [\cP_k(\cT_h)]^{ d\times d} \cap \mathbb H(\bdiv, \Omega), \quad 
	\mathbb{Q}_h := [\cP_{k-1}(\cT_h)]^{d\times d} \cap \mathbb{Q}, \quad \text{and} \quad \mathbf{U}_h := [\cP_{k-1}(\cT_h)]^{d}.
\end{equation*}
The triplet $\set{ \mathbb{W}_h, \mathbb{Q}_h, \mathbf{U}_h}$ constitutes the Arnold-Falk-Winther family introduced in \cite{ArnoldFalkWinther} for the steady elasticity problem. It is shown in \cite{ArnoldFalkWinther} that $\big(\btau, (\bs, \bv) \big)\mapsto (\bs,\btau) + (\bv,\bdiv \tau)$ satisfies a uniform inf-sup condition for the pair $\{\mathbb W_h,   \mathbb{Q}_h \times\mathbf U_h\}$. Therefore, if we let 
\[
\mathfrak{S}_h^{DG} := [\cP_{k}(\cT_h)]^{d\times d} \times [\cP_k^V(\cT_h)]^{ d\times d} 
\quad \text{and} \quad 
\mathfrak S_h:=\set{\kq\in \mathfrak{S}_h^{DG}:\quad \jmath_\omega^+\kq \in \mathbb{W}_h } \subset \mathfrak{S}, 
\]
the embedding $\mathbb W_h\times \{\mathbf 0\} \hookrightarrow \mathfrak{S}_h $ implies the existence of $\beta^*>0$, independent of $h$, such that 
\begin{equation}\label{discInfSup}
\sup_{\kq \in \mathfrak S_h} 
\frac{(\bs,\jmath_\omega^+\kq) \,+\, \big(\bv,\bdiv \jmath_\omega^+\kq\big)}{\norm{\kq}_{\mathfrak S}}
\,\ge\, \beta^* \Big\{ \norm{\bs}_{0,\Omega} + \norm{\bv}_{0,\Omega} \Big\},\quad \forall (\bs,\bv) \in \mathbb{Q}_h \times\mathbf{U}_h.
\end{equation} 

Given $K\in \cT_h$, we denote by $\Pi_K: \mathbb H^1(K) \to [\cP_k(K)]^{d\times d}$ the tensorial version   of the (local) BDM-interpolation operator and recall the following classical error estimate, see \cite[Proposition 2.5.4]{BoffiBrezziFortinBook}, 
\begin{equation}\label{asymp}
 \norm{\btau - \Pi_K \btau}_{0,K} \leq C h_K^{m} \norm{\btau}_{m,K} \quad \forall \btau \in  \mathbb H^m(K) \quad \text{with $1 \leq m \leq k+1$}.
\end{equation}
Moreover, thanks to the commutative property, if $\bdiv \btau \in \bH^m(K)$, then
\begin{equation}\label{asympDiv}
 \norm{\bdiv (\btau - \Pi_K \btau) }_{0,K} = \norm{\bdiv \btau -  U_K \bdiv \btau }_{0,K} 
 \leq C h_K^{m}  \norm{\bdiv\btau}_{m,K},
\end{equation} 
for   $0 \leq m \leq k$, where $U_K$ is the $\bL^2(K)$-orthogonal projection onto $[\cP_{k-1}(K)]^d$. Let us consider now the global interpolation operator $\Pi_h: \mathbb H^1(\cT_h) \to [\cP_{k}(\cT_h)]^{d\times d}$ defined piecewise by $(\Pi_h\btau)|_K := \Pi_K \btau_K $, for all $\btau=\set{\btau_K}\in \mathbb H^1(\cT_h)$. We point out that $\Pi_h: \mathbb H^1(\cT_h)\cap \mathbb H(\bdiv, \Omega) \to \mathbb{W}_h$ and 
\[
\bdiv \Pi_h \btau = U_h \bdiv \btau,\quad \forall \btau \in \mathbb H^1(\cT_h)\cap \mathbb H(\bdiv, \Omega),
\]
where $(U_h\bv)|_K = U_K(\bv|_K)$ for all $\bv\in \bL^2(\Omega)$. Consequently, 
$\boldsymbol{\Pi}_h\kq := (\Pi_h \bet, \Pi_h\btau)\in \mathfrak S_h$ for all $\kq=(\bet,\btau) \in [\mathbb H^1(\cT_h)]^2\cap \mathfrak S$ since $\jmath_\omega^+ \boldsymbol{\Pi}_h\kq = \Pi_h (\bet + \omega \btau) \in \mathbb W_h$. Moreover, we deduce from \eqref{asymp} and \eqref{asympDiv} that if  $\kq \in [\mathbb H^m(\cT_h)]^2\cap \mathfrak S$ and $\jmath_\omega^+ \kq \in \bbH^m(\cT_h)$ then 
\begin{equation}\label{interpEs}
	\norm{\kq - \boldsymbol{\Pi}_h \kq}_{\mathfrak S} \lesssim C h^m \big( \norm{\kq}_{ [\mathbb H^m(\cT_h)]^2} + \norm{\bdiv \jmath_\omega^+ \kq}_{\bbH^m(\cT_h)} \big),
\end{equation}
where $\norm{\btau}^2_{\bbH^m(\cT_h)}:=\sum_{K\in \cT_h} \norm{\btau_K}^2_{m,K}$ for all $\btau \in \bbH^m(\cT_h)$.


We will now construct an auxiliary operator $\varXi_h:\ \mathfrak{S} \to \mathfrak{S}_h$ that will play a central role in our analysis.  It is defined by $\varXi_h\kp := \widetilde\kp_h$, where $(\widetilde\kp_h, \widetilde\br_h, \widetilde\bu_h)\in \mathfrak{S}_h \times \mathbb{Q}_h \times \mathbf U_h$ is the solution of 
\begin{align}\label{Xih}
\begin{split}
\big(\widetilde \kp_h, \widetilde \kq\big) + (\widetilde\br_h, \jmath_\omega^+\kq) 
+ ( \widetilde\bu_h, \bdiv \jmath_\omega^+\kq )_\rho  &= \big( \kp , \kq\big),\quad \forall \kq \in \mathfrak{S}_h
\\[1ex]
(\bs, \jmath_\omega^+\widetilde\kp_h) + (\bv,\bdiv \jmath_\omega^+ \widetilde \kp_h)_\rho &= (\bs, \jmath_\omega^+\kp) + \big(\bv, \bdiv \jmath_\omega^+\kp \big)_\rho, \quad  \forall (\bs, \bv) \in \mathbb{Q}_h\times \mathbf U_h.
\end{split}
\end{align}
It is important to notice that  the kernel of the bilinear form $\mathfrak{S}_h\times \mathbb{Q}_h \ni (\kq, \bs) \mapsto (\bs , \jmath_\omega^+\kq)$, namely 
\[
\mathfrak{S}_{\text{sym},h} := \big\{ \kq\in \mathfrak{S}_h:\ (\bs,\jmath_\omega^+ \kq) = 0 \quad \forall \,\bs \in \mathbb{Q}_h \big\},
\]
is not a subspace of $\mathfrak{S}_{\text{sym}}$. We denote by $\mathfrak{K}:= \big\{ \kq\in \mathfrak{S}_{\text{sym}}: \  \bdiv \jmath_\omega^+\kq = \mathbf 0 \big\}$ the kernel of the bilinear form $\mathfrak S \times (\mathbb Q \times \mathbf L^2(\Omega)) \ni \big(\kq, (\bs, \bv) \big)\mapsto (\bs,\jmath_\omega^+\kq) + (\bv,\bdiv \jmath_\omega^+\kq)$. It discrete counterpart is given by $\mathfrak{K}_{h} := \big\{ \kq\in \mathfrak{S}_{\text{sym,h}}:\  \bdiv \jmath_\omega^+\kq = \mathbf 0 \big\}$. The inf-sup condition \eqref{InfSupC} and the fact that $(\kq,\kq) = \norm{\kq}_{0,\Omega}^2 = \norm{\kq}^2_{\mathfrak S}$ for all $\kq \in \mathfrak{K}$ permits us to apply the Babu\v{s}ka-Brezzi theory to prove that the continuous counterpart of problem \eqref{Xih} is well-posed. Its unique solution is easily seen to be $(\kp, \mathbf 0, \mathbf 0) \in \mathfrak S \times \mathbb Q \times \mathbf L^2(\Omega)$. In turn, noting that we also have $(\kq,\kq) = \norm{\kq}^2_{\mathfrak S}$ for all $\kq \in \mathfrak{K}_h$, and employing now the inf-sup condition \eqref{discInfSup} and the discrete Babu\v{s}ka-Brezzi theory, we deduce that problem \eqref{Xih} is well-posed uniformly in $h$. Therefore, C\'ea's estimate between $(\kp, \mathbf 0, \mathbf 0)$ and $(\widetilde\kp_h, \widetilde\br_h, \widetilde\bu_h)$ implies the following approximation property for $\varXi_h$:
\begin{equation}\label{XihStab}
	\norm{\kp - \varXi_h\kp}_{\mathfrak S} \lesssim
\inf_{\kq_h\in \mathfrak{S}_h}\norm{\kp-\kq_h}_{\mathfrak{S}} \quad \forall \kp\in \mathfrak S.
\end{equation}
Moreover,  taking into account  the compatibility of $\mathcal{T}_h$ with the partition $\bar\Omega = \cup_{j= 1}^J \widehat{\Omega}_j$, it turns out that  $\varXi_h$ satisfies by construction the commuting diagram property 
\begin{equation}\label{div0}
\tfrac{1}{\rho}\bdiv \jmath_\omega^+(\varXi_h\kp) = U_h (\tfrac{1}{\rho} \bdiv\jmath_\omega^+\kp) , \quad \forall \kp \in \mathfrak S.
\end{equation}

We end this section by introducing notations related to DG approximations of $\H(\text{div})$-type spaces. We say that a closed subset $F\subset \overline{\Omega}$ is an interior edge/face if $F$ has a positive $(d-1)$-dimensional measure and if there are distinct elements $K$ and $K'$ such that $F =\bar K\cap \bar K'$. A closed subset $F\subset \overline{\Omega}$ is a boundary edge/face if there exists $K\in \cT_h$ such that $F$ is an edge/face of $K$ and $F =  \bar K\cap \partial \Omega$. We consider the set $\cF_h^0$ of interior edges/faces, the set $\cF_h^\partial$ of boundary edges/faces, and let  $\cF_h := \cF_h^0\cup \cF_h^\partial$. For any element $K\in \cT_h$, we introduce the set 
 \[
 \cF(K):= \set{F\in \cF_h:\quad F\subset \partial K} 
 \]
 of edges/faces composing the boundary of $K$.


 We will need the space given on the skeletons of the triangulations $\cT_h$  by $\L^2(\cF^0_h):= \bigoplus_{F\in \mathcal{F}^0_h} \L^2(F)$. Its vector valued version is denoted $\bL^2(\cF^0_h):= [\L^2(\cF^0_h)]^d$. Here again, the components $\bv_F$ of $\bv := \set{\bv_F}\in \bL^2(\cF^0_h)$  coincide with the restrictions $\bv|_F$.  We endow $\bL^2(\cF^0_h)$ with the inner product 
\[
(\bu, \bv)_{\cF^0_h} := \sum_{F\in \cF^0_h} \int_F \bu_F\cdot \bv_F\quad \forall \bu,\bv\in \bL^2(\cF^0_h)
\]
and denote the corresponding norm $\norm{\bv}^2_{0,\cF^0_h}:= (\bv,\bv)_{\cF^0_h}$. From now on, $h_\cF\in \L^2(\cF^0_h)$ is the piecewise constant function defined by $h_\cF|_F := h_F$ for all $F \in \cF^0_h$ with $h_F$ denoting the diameter of edge/face $F$. 

Given  $\bv\in \mathbf H^s(\cT_h)$ and $\btau\in \mathbb H^s(\cT_h)$, with $s>1/2$, we define averages $\mean{\bv}\in \bL^2(\cF^0_h)$ and jumps $\jump{\btau}\in \bL^2(\cF^0_h)$ by
\[
 \mean{\bv}_F := (\bv_K + \bv_{K'})/2 \quad \text{and} \quad \jump{\btau}_F := 
 \btau_K \bn_K + \btau_{K'}\bn_{K'} 
 \quad \forall F \in \cF(K)\cap \cF(K'),
\]
where $\bn_K$ is the outward unit normal vector to $\partial K$. 

Finally, we recall the following discrete trace inequality. 
\begin{prop}\label{card} 
There exists a constant $C_{\text{tr}}>0$ independent of $h$ such that 
 \begin{equation}\label{discTrace}
  \norm{h^{1/2}_{\cF}\mean{v}}_{0,\cF^0_h}\leq C_{\text{tr}} \norm{v}_{0,\Omega}\quad \forall  v\in \cP_k(\cT_h). 
 \end{equation}
\end{prop}
\begin{proof}
See \cite[Proposition 4.1]{conThan}.
\end{proof}

\section{The semi-discrete CG problem and its convergence analysis}
We consider the following semi-discrete counterpart of  \eqref{varFormR1-varFormR2}: Find $\kp_h\in \cC^1(\mathfrak{S}_h)$ and $\br_h\in \cC^0(\mathbb{Q}_h)$ solving 
\begin{align}\label{varFormR1-varFormR2-hc}
\begin{split}
\dfrac{\text{d}}{\text{d}t}  \Big\{ A\big( \jmath_\omega\dot\kp_h + \pi_2\kp_h , \jmath_\omega\kq\big)  + (\br_h,\jmath_\omega^+\kq)  \Big\}
+ \big( \bdiv\jmath_\omega^+\kp_h + \bF,  \bdiv \jmath_\omega^+ \kq \big)_\rho & =   \big<\ddot{\boldsymbol{g}}, \jmath_\omega^+\kq\bn \big>_\Gamma, 
\quad \forall \kq\in \mathfrak{S}_h
\\[1ex]
 (\bs, \jmath_\omega^+\kp_h)  &=  0 \quad \forall \bs\in \mathbb{Q}_h. 
 \end{split}
 \end{align}

The elliptic projector $\varXi_h:\mathfrak S\to \mathfrak S_h$ introduced in Section~\ref{section4} will allow us to apply standard techniques of error analysis to our scheme. From now on we assume that $\kp_1\in \mathfrak{S}_{\text{sym}}$ and we consider  a solution $(\kp_h(t),\br_h(t))$ of problem \eqref{varFormR1-varFormR2-hc} started up with the initial conditions 
\begin{equation}\label{initial-R1-R2-h*c}
\kp_h(0)= \varXi_h\kp_0,
\quad  \dot\kp_h(0) = \varXi_h\kp_1, \quad \br_h(0) = \mathbf 0.
\end{equation}
In this way, the projected errors $\be_{\kp,h}(t) := \varXi_h\kp(t) - \kp_h(t)$ and $\be_{r,h}(t) := Q_h\br(t) - \br_h(t)$  satisfy by construction vanishing initial conditions: 
\begin{equation}\label{inti0c}
\be_{\kp, h}(0)=(\mathbf 0,\mathbf 0)\quad  \dot{\be}_{\kp, h}(0)= (\mathbf 0,\mathbf 0), \quad \text{and} \quad \be_{\br,h}(0) = \mathbf{0}.
\end{equation}
Moreover, by definition of $\varXi_h$ and due to the second equations of \eqref{varFormR1-varFormR2} 
and \eqref{varFormR1-varFormR2-hc}, it turns out that 
\begin{equation}\label{skew0c}
	\big(\bs, \jmath_\omega^+\be_{\kp, h}(t)\big) = \big(\bs, \jmath_\omega^+\kp(t)\big) - \big(\bs, \jmath_\omega^+\kp_h(t)\big)= 0,\quad \forall \bs \in 
	\mathbb{Q}_h.
\end{equation}  
\begin{theorem}\label{errorEc} 
Assume that the solution of problem \eqref{varFormR1-varFormR2} satisfies  the regularity assumptions $\kp\in \cC^2(\mathfrak{S})$ and $\br\in \cC^1(\mathbb Q)$. Then, the following error estimate holds 
\begin{equation}
\begin{array}{c}
\max_{t\in [0, T]}\norm{(\kp-\kp_h)(t)}_{\mathfrak{S}} 
\,+\, \max_{t\in [0, T]}\norm{\jmath_\omega (\dot\kp-\dot\kp_h)(t)}_{\mathfrak{L}_V^2(\Omega)} + \max_{t\in [0, T]}\norm{ (\br - \br_h)(t)}_{0,\Omega} 
\\[2ex]
 \lesssim \, \norm{\kp - \varXi_h\kp}_{\W^{2,\infty}(\mathfrak{S})}+
\norm{\br - Q_h\br}_{\W^{1,\infty}(\mathbb{L}^2(\Omega))}.
\end{array}
\end{equation}
\end{theorem}
\begin{proof}
We first observe that,  because of the regularity assumptions on $\kp$ and $\br$, we have
\begin{equation}\label{idempoic}
\begin{array}{l}
 
\frac{\text{d}}{\text{d}t} Q_h\br(t) = Q_h \dot\br(t) \quad\text{and} \quad  \frac{\text{d}^i}{\text{d}t^i}\varXi_h\kp(t) = \varXi_h\frac{\text{d}^i}{\text{d}t^i}\kp(t),  \quad \forall\, i \in \{1,2\}\,, \quad \forall\, t\in [0,T].
\end{array}
\end{equation} 
Then, using that the scheme \eqref{varFormR1-varFormR2-hc} is consistent with  \eqref{varFormR1-varFormR2}, and keeping in mind \eqref{div0} together with the fact that $\bdiv \jmath_\omega^+ \mathfrak S_h \subset \mathbf U_h$, we readily find  
\begin{align}\label{errorIdentityc}
\begin{split}
A\big( \jmath_\omega\ddot{\be}_{\kp, h} +  \pi_2\dot{\be}_{\kp, h}, \jmath_\omega\kq\big)  + (\dot{\be}_{r,h}, \jmath_\omega^+ \kq)   
+ \big(\bdiv \jmath_\omega^+ \be_{\kp,h}(t),  \bdiv\jmath_\omega^+\kq \big)_\rho&= F(\kq), \quad \forall \kq \in \mathfrak S_h,
\end{split}
\end{align}
with
\[
F\big(\kq\big):=  A\big( \jmath_\omega (\varXi_h \ddot\kp - \ddot\kp) + \pi_2(\varXi_h \dot\kp - \dot\kp), \jmath_\omega\kq\big)  +(Q_h\dot{\br} - \dot{\br}, \jmath_\omega^+\kq).
\]
Next, choosing  $\kq = \dot{\be}_{\kp, h}(t)$ in  \eqref{errorIdentityc} and taking into account \eqref{skew0c}, we deduce that 
\begin{equation}
	\dot{\mathcal{E}}\big({\be}_{\kp, h}\big)(t) + (\omega \cV\pi_2\dot{\be}_{\kp, h},  \pi_2\dot{\be}_{\kp, h}  )
	  = F(\dot{\be}_{\kp, h}).
\end{equation}
Hence,  as the second term on the left-hand side is nonnegative, the Cauchy-Schwarz inequality combined with \eqref{eq-extra-1}  give   
\[
\dfrac{\dot{\mathcal{E}}\big({\be}_{\kp, h}\big)}{2\sqrt{{\mathcal{E}}\big({\be}_{\kp, h}\big)}}
 \lesssim   \sum_{i=1,2}\norm{ \frac{\text{d}^i\kp}{\text{d}t^i}  - \varXi_h\frac{\text{d}^i\kp}{\text{d}t^i}}_{\mathfrak{L}_V^2(\Omega)} + 
 \norm{\dot{\br} - Q_h\dot{\br}}_{0,\Omega},
\]
and integrating with respect to time  we arrive at 
\begin{equation}\label{8.1c}
\max_{t\in [0, T]}{\mathcal{E}}\big({\be}_{\kp, h}\big)^{1/2}(t) \lesssim  \norm{\kp-\varXi_h\kp}_{\W^{2,\infty}(\mathfrak{L}_V^2(\Omega))} + \norm{\br - Q_h\br}_{\W^{1,\infty}(\mathbb{L}^2(\Omega))}.
\end{equation}
It follows now from  \eqref{eq-extra-1} and $\be_{\kp, h}(t) = \int_0^t \dot{\be}_{\kp, h}(s)\text{d}s$ that 
\begin{equation}\label{har1c}
\max_{t\in [0, T]}\norm{ {\be}_{\kp, h}(t) }_{\mathfrak S} 
+ \max_{t\in [0, T]}\norm{\jmath_\omega\dot \be_{\kp, h}(t) }_{\mathfrak{L}_V^2(\Omega)} \lesssim \norm{\kp-\varXi_h\kp}_{\W^{2,\infty}(\mathfrak{L}_V^2(\Omega))} + \norm{\br - Q_h\br}_{\W^{1,\infty}(\mathbb{L}^2(\Omega))}.
\end{equation}

In order to estimate the error in the rotation $\br$ we notice that 
integrating once with respect to time in  \eqref{errorIdentityc} we obtain
\begin{align}\label{errorIdentityr0c}
\begin{split}
(\be_{\br,h}, \jmath_\omega^+\kq) &= -
A\big( \jmath_\omega \dot{\be}_{\kp,h}  + \pi_2 \be_{\kp,h}, \jmath_\omega\kq\big)   
-  \int_0^t \big(\bdiv\jmath_\omega^+\be_{\kp,h}(s),  \bdiv\jmath_\omega^+\kq \big)_\rho\, \text{d}s,
\\[1ex]
&\quad -A\big( \jmath_\omega(\dot\kp - \varXi_h \dot\kp) + \pi_2 (\kp - \varXi_h\kp), \jmath_\omega\kq\big) + A\big( \jmath_\omega(\kp_1 - \varXi_h \kp_1) + \pi_2 (\kp_0 - \varXi_h\kp_{0}), \jmath_\omega\kq\big) 
\\[1ex]
&\quad - (\br - Q_h\br, \jmath_\omega^+\kq),\quad \forall \kq \in \mathfrak S_h.
\end{split}
\end{align}
Therefore, the inf-sup condition \eqref{discInfSup}, identity \eqref{errorIdentityr0c},  the Cauchy-Schwarz inequality, \eqref{contA} and \eqref{har1c} yield  
\begin{equation}\label{hartc}
	\beta^*\norm{\be_{\br,h} }_{0,\Omega} \leq \sup_{\kq\in \mathfrak S_h}
\dfrac{(\be_{\br,h}, \jmath_\omega^+\kq )}{\norm{\kq}_{\mathfrak S}}
\lesssim 
\norm{\kp-\varXi_h\kp}_{\W^{2,\infty}(\mathfrak{L}_V^2(\Omega))} + \norm{\br - Q_h\br}_{\W^{1,\infty}(\mathbb{L}^2(\Omega))}.
\end{equation}
Finally, the splittings  $\kp - \kp_h = (\kp - \varXi_h\kp) + \be_{\kp, h}$ and $\br - \br_h = (\br - Q_h\br) + \be_{\br, h}$ of each component the error, the triangle inequality, together with   \eqref{hartc} and \eqref{har1c} give  
\begin{align}\label{eqBbc}
\begin{split}
\max_{[0, T]}\norm{\kp - \kp_h}_{\mathfrak{S}} 
+ \max_{[0, T]}\norm{\jmath_\omega(\dot\kp-\dot\kp_h)}_{\mathfrak{L}_V^2(\Omega)} + \max_{t\in [0, T]}\norm{\br-\br_h}_{0,\Omega} 
\\[1ex]
\lesssim \,\norm{\kp - \varXi_h\kp}_{\W^{2,\infty}(\mathfrak{S})}
 + \norm{\br - Q_h\br}_{\W^{1,\infty}(\mathbb{L}^2(\Omega))},
\end{split}
\end{align}
and the result follows.
\end{proof}

\begin{corollary}\label{coro1c}
Assume that the solution of problem \eqref{varFormR1-varFormR2}  is such that  $\kp \in \cC^2(\mathbb{H}^{k}(\Omega)^{2})$, $\bdiv \kp^+ \in \cC^2(\bH^{k}(\Omega))$ and $\br\in \cC^1(\mathbb{H}^{k}(\Omega))$. Then there holds
\begin{equation}\label{asympSDc}
\max_{[0, T]}\norm{\jmath_\omega(\dot\kp-\dot\kp_h)(t)}_{\mathfrak{L}_V^2(\Omega)} + \max_{[0, T]}\norm{\kp(t)-\kp_h(t)}_{\mathfrak S} +\norm{\br-\br_h}_{\W^{1,\infty}(\mathbb{L}^2(\Omega))}   \lesssim  h^k.
\end{equation}
\end{corollary}
\begin{proof}
It is a direct consequence of \eqref{asymp}, \eqref{asympDiv}, \eqref{XihStab}, and Theorem \ref{errorEc}.
\end{proof}

We stress here that the quantity $\ddot{\bu}_h(t) := \tfrac{1}{\rho}\big(\bdiv \kp_h^+(t) + U_h \bF(t)\big)$  provides a direct and accurate approximation of the acceleration field $\ddot \bu$. Indeed, under the assumptions of Corollary~\ref{coro1c},  the triangle inequality yields
\[
\max_{t\in [0, T]}\| (\ddot\bu - \ddot \bu_h)(t)\|_{0,\Omega}\lesssim \max_{t\in [0, T]}\|\bdiv \jmath_\omega^+(\kp - \kp_h) \|_{0,\Omega} + \max_{t\in [0, T]}\|\bF - U_h \bF\|_{0,\Omega} \lesssim h^k.
\]

\section{The semi-discrete DG problem and its convergence analysis}\label{section5}

For any $k\geq 1$, we recall that $\mathfrak S^{DG}_h = [\cP_k(\cT_h)]^{d\times d}\times [\cP_k(\cT_h^V)]^{ d\times d}$ and let $\mathfrak S(h) :=\mathfrak S + \mathfrak S^{DG}_h$.  Given $\kq= (\bet, \btau) \in \mathfrak S^{DG}_h$, we define $\bdiv_h \jmath_\omega^+\kq \in \mathbf L^2(\Omega)$ by $(\bdiv_h \jmath_\omega^+\kq)|_{K} := \bdiv (\bet_K + \omega\btau_K)$ for all $K\in \cT_h$ and endow $\mathfrak S(h)$ with the norm
\[
 \norm{\kq}^2_{\mathfrak S(h)} := \norm{\jmath_\omega\kq}^2_{\mathfrak{L}_V^2(\Omega)} + \norm{\bdiv_h \jmath_\omega^+ \kq}^2_{0,\Omega} + \norm{h_{\cF}^{-1/2} \jump{\kq}}^2_{0,\cF^0_h}.
\]

From now on we assume that there exists $s > 1/2$ such that $\bF|_{\Omega_j}\in \bH^s(\Omega_j)$,  for $j = 1,\ldots ,J$. We consider the following semi-discrete counterpart of  \eqref{varFormR1Extended}: Find $\kp_h\in \cC^1(\mathfrak S_h^{DG})$ and $\br_h\in \cC^0(\mathbb Q_h)$ solving 
\begin{align}\label{varFormR1-varFormR2-h}
\begin{split}
 A\big( \jmath_{\omega}\ddot\kp_h &+ \pi_2 \dot\kp_h , \jmath_{\omega}\kq\big)  + (\dot{\br}_h,\jmath_{\omega}^+\kq)  + 
\big( \bdiv_h\jmath_{\omega}^+\kp_h,  \bdiv_h \jmath_{\omega}^+\kq \big)_\rho
\\[1ex]
&- \big(\mean{\tfrac{1}{\rho} \bdiv_h \jmath_\omega^+ \kp_h}, \jump{\jmath_\omega^+ \kq}\big)_{\cF^0_h}
- \big(\mean{\tfrac{1}{\rho} \bdiv_h \jmath_\omega^+ \kq}, \jump{\jmath_\omega^+ \kp_h}\big)_{\cF^0_h}
\\[1ex]
&\qquad + \big(\texttt{a} h_\cF^{-1}\jump{\jmath_\omega^+ \kp_h}, \jump{\jmath_\omega^+ \kq} \big)_{\cF^0_h}
 = -\big(\bF,  \bdiv_h \jmath_{\omega}^+\kq \big)_\rho + \big(\mean{\tfrac{1}{\rho} \bF}, \jump{\jmath_\omega^+ \kq}\big)_{\cF^0_h}+ (\ddot{\boldsymbol{g}}, \jmath_\omega^+\kq\bn)_{\cF_h^\partial}, 
 \\[1ex]
 &\qquad \qquad \qquad \qquad \quad (\bs, \jmath^+_{\omega}\kp_h) = 0, 
 \end{split}
 \end{align}
 for all $\kq\in \mathfrak{S}_h$ and $\bs\in \mathbb{Q}_h$, and subject to the initial conditions 
 \eqref{initial-R1-R2-h*c}. Here, we are using the notation
 \[
 (\ddot{\boldsymbol{g}}, \jmath_\omega^+\kq\bn)_{\cF_h^\partial} := \sum_{F\in \cF_h^\partial} \int_F \ddot{\boldsymbol{g}} \cdot (\jmath_\omega^+\kq)|_{F} \bn.
 \]

We now need to verify that, under suitable regularity assumptions, the DG scheme \eqref{varFormR1-varFormR2-h} is consistent with problem \eqref{varFormR1-varFormR2}. 

\begin{prop}\label{consistency}
Let  $(\kp, \br)$ be  the solution of \eqref{varFormR1Extended}. We assume that  $\kp\in \cC^2(\mathfrak L^2(\Omega))$, $\br\in \cC^1(\mathbb Q)$, and $\jmath_\omega^+\kp\in \mathbb{H}^s(\cT_h)$, with $s>1/2$. Then,  it holds 
 \begin{align}\label{consistent}
\begin{split}
 A\big( \jmath_{\omega}\ddot\kp + \pi_2 \dot\kp , \imath_{\omega}\kq\big)  &+ (\dot{\br},\jmath_{\omega}^+\kq)  
+ \big( \bdiv_h\jmath_{\omega}^+\kp,  \bdiv_h \jmath_{\omega}^+\kq \big)_\rho + \big(\texttt{a} h_\cF^{-1}\jump{\jmath_\omega^+ \kp}, \jump{\jmath_\omega^+ \kq} \big)_{\cF^0_h}
\\[1ex]
&- \big(\mean{\tfrac{1}{\rho} \bdiv_h \jmath_\omega^+ \kp}, \jump{\jmath_\omega^+ \kq}\big)_{\cF^0_h}
- \big(\mean{\tfrac{1}{\rho} \bdiv_h \jmath_\omega^+ \kq}, \jump{\jmath_\omega^+ \kp}\big)_{\cF^0_h}
\\[1ex]
 &\qquad\quad \quad  =   
-\big(  \bF,  \bdiv_h \jmath_{\omega}^+\kq \big)_\rho + \big(\mean{ \tfrac{1}{\rho} \bF}, \jump{\jmath_\omega^+ \kq}\big)_{\cF^0_h}+ (\ddot{\boldsymbol{g}}, \jmath_\omega^+\kq\bn)_{\cF_h^\partial},\quad \forall \kq\in \mathfrak S^{DG}_h,
\\[1ex]
&\quad  (\bs, \jmath^+_{\omega}\kp)  = 0, \quad  \forall \bs\in \mathbb{Q}_h,
 \end{split}
 \end{align}
\end{prop}
\begin{proof}
The second equation of \eqref{consistent} is satisfied because $\mathbb Q_h \subset \mathbb Q$. On the other hand, using that $\jump{\jmath_\omega^+ \kp} = \mathbf 0$ and $ \rho^{-1} \big( \bdiv_h \jmath_\omega^+ \kp + \bF \big) = \rho^{-1} \big(  \bdiv \jmath_\omega^+ \kp + \bF \big) = \ddot\bu$ yield 
\begin{align}\label{consistent0}
\begin{split}
 & A\big( \jmath_{\omega}\ddot\kp + \pi_2 \dot\kp , \jmath_{\omega}\kq\big)  + (\dot{\br},\jmath_{\omega}^+\kq)  
+ \big( \bdiv_h\jmath_{\omega}^+\kp,  \bdiv_h \jmath_{\omega}^+\kq \big)_\rho
\\[1ex]
&\quad - \big(\mean{\tfrac{1}{\rho} \bdiv_h \jmath_\omega^+ \kp}, \jump{\jmath_\omega^+ \kq}\big)_{\cF^0_h}
- \big(\mean{\tfrac{1}{\rho} \bdiv_h \jmath_\omega^+ \kq}, \jump{\jmath_\omega^+ \kp}\big)_{\cF^0_h}
+ \big(\texttt{a} h_\cF^{-1}\jump{\jmath_\omega^+ \kp}, \jump{\jmath_\omega^+ \kq} \big)_{\cF^0_h}
\\[1ex]
 &\qquad \quad =   A\big( \jmath_{\omega}\ddot\kp + \pi_2 \dot\kp , \imath_{\omega}\kq\big)  + (\dot{\br},\jmath_{\omega}^+\kq)  
+ \big( \ddot\bu,  \bdiv_h \jmath_{\omega}^+\kq \big) - \big(\mean{\ddot\bu} , \jump{\jmath_\omega^+ \kq}\big)_{\cF^0_h}
\\[1ex]
&
\qquad \qquad \qquad \qquad \qquad \qquad \quad -\big(  \bF,  \bdiv_h \jmath_{\omega}^+\kq \big)_\rho + \big(\mean{ \tfrac{1}{\rho} \bF}, \jump{\jmath_\omega^+ \kq}\big)_{\cF^0_h},\quad \forall \kq\in \mathfrak S^{DG}_h
 \end{split}
 \end{align}
 Now, taking into account that ($\bu|_\Gamma = \boldsymbol{g}$)
 \begin{align*}
 	\big(\mean{\ddot\bu} , \jump{\jmath_\omega^+ \kq}\big)_{\cF^0_h} &= \big(\mean{\ddot\bu} , \jump{\jmath_\omega^+ \kq}\big)_{\cF_h} - (\ddot{\boldsymbol{g}}, \jmath_\omega^+\kq\bn)_{\cF_h^\partial} = \sum_{K\in \cT_h}\int_{\partial K} \ddot{\bu}\cdot (\jmath_\omega^+ \kq)\bn_K - (\ddot{\boldsymbol{g}}, \jmath_\omega^+\kq\bn)_{\cF_h^\partial}
 	\\[1ex]
 	&= \sum_{K\in \cT_h} \int_K \big( \nabla \ddot{\bu}: \jmath_\omega^+\kq + \ddot{\bu}\cdot \bdiv \jmath_\omega^+\kq\big)- (\ddot{\boldsymbol{g}}, \jmath_\omega^+\kq\bn)_{\cF_h^\partial},
 \end{align*}
 and keeping in mind \eqref{rmeans} and \eqref{const+}, we deduce that  
 \begin{align*}\label{plug}
 \begin{split}
 \big( \ddot\bu,  \bdiv_h \jmath_{\omega}^+\kq \big) - \big(\mean{\ddot\bu} , \jump{\jmath_\omega^+ \kq}\big)_{\cF^0_h} &= - (\beps(\ddot \bu), \jmath_{\omega}^+\kq) - (\dot\br, \jmath_{\omega}^+\kq) + \dual{\ddot{\boldsymbol{g}}, \jmath_\omega^+ \kq}_\Gamma
 \\
 & = -A\big( \jmath_{\omega}\ddot\kp + \pi_2 \dot\kp , \jmath_{\omega}\kq\big)  -(\dot{\br},\jmath_{\omega}^+\kq) + (\ddot{\boldsymbol{g}}, \jmath_\omega^+\kq\bn)_{\cF_h^\partial}.
 \end{split}
 \end{align*}
 Plugging  the last identity in \eqref{consistent0} gives the result. 
\end{proof}


Here again, we decompose the errors $\kp - \kp_h = \mathcal I_{\kp}^h + \be_{\kp}^h$ and $\br - \br_h = \mathcal I_{\br}^h + \be_{\br}^h$,  with $\be_{\kp}^h(t) :=  \varXi_h\kp(t) - \kp_h(t)$ and $\be_{\br}^h(t) :=  Q_h\br(t) - \br_h(t)$, so that \eqref{inti0c} and \eqref{skew0c} still hold true.   
We recall that $\mathcal I_{\kp}^h := \kp - \varXi_h \kp \in \mathfrak S$ by definition of $\varXi_h$.  

\begin{theorem} \label{convEstim}
Assume that the hypotheses of Proposition~\ref{consistency} are satisfied. There exists a constant $\emph{\texttt{a}}_0>0$, independent of $h$, such that  the  error estimate 
\begin{equation}\label{errorE}
\begin{array}{rc}
\max_{t\in [0, T]}\norm{(\kp-\kp_h)(t)}_{\mathfrak{S}(h)} + \max_{t\in [0, T]}\norm{\jmath_\omega (\dot\kp-\dot\kp_h)(t)}_{\mathfrak{L}_V^2(\Omega)} + \max_{t\in [0, T]}\norm{ (\br - \br_h)(t)}_{0,\Omega}
\\[1ex]
\lesssim \norm{\mathcal{I}_{\kp}^h }_{\W^{2,\infty}(\mathfrak{S})} + \norm{ h_\cF^{1/2} \mean{\tfrac{1}{\rho} \bdiv \jmath_\omega^+ {\mathcal I}_{\kp}^h}}_{W^{1,\infty}(\bL^2(\cF^0_h))}+ \norm{\mathcal{I}_{\br}^h}_{\W^{1,\infty}(\mathbb{L}^2(\Omega))},
\end{array}
\end{equation}
holds true for all $\emph{\texttt{a}}\geq \emph{\texttt{a}}_0$.
\end{theorem}
\begin{proof}
Using \eqref{consistent}, \eqref{div0} combined with the fact that $\bdiv_h \big(\jmath_\omega^+ \mathfrak S_h^{DG}\big) \subset \mathbf U_h$, we obtain the identity  
\begin{align}\label{errorIdentity}
\begin{split}
A\big( \jmath_\omega\ddot{\be}_{\kp}^h &+  \pi_2\dot{\be}_{\kp}^h, \jmath_\omega\kq\big)  + (\dot{\be}_{\br}^h, \jmath_\omega^+\kq)    
\\[1ex]
&+ \big(\bdiv_h \jmath_\omega^+\be_{\kp}^h(t),  \bdiv_h\jmath_\omega^+\kq \big)_\rho + \big(\texttt{a} h_\cF^{-1}\jump{\jmath_\omega^+ \be_{\kp}^h}, \jump{\jmath_\omega^+ \kq} \big)_{\cF^0_h}
\\[1ex]
&- \big(\mean{\tfrac{1}{\rho} \bdiv_h \jmath_\omega^+ \be_{\kp}^h}, \jump{\jmath_\omega^+ \kq}\big)_{\cF^0_h} - \big(\mean{\tfrac{1}{\rho} \bdiv_h \jmath_\omega^+ \kq}, \jump{\jmath_\omega^+ \be_{\kp}^h}\big)_{\cF^0_h} = F(\kq), \quad \forall \kq \in \mathfrak S^{DG}_h,
\end{split}
\end{align}
where 
\[
F\big(\kq\big):=  - A\big( \jmath_\omega \ddot{\mathcal I}_{\kp}^h + \pi_2\dot{\mathcal I}_{\kp}^h, \jmath_\omega\kq\big)  - (\dot{\mathcal I}_{\br}^h , \jmath_\omega^+\kq) + (\mean{\tfrac{1}{\rho} \bdiv \jmath_\omega^+ {\mathcal I}_{\kp}^h}, \jump{\jmath_\omega^+\kq})_{\cF^0_h}.
\]
Thanks to \eqref{skew0c},   the choice $\kq = \dot{\be}_{\kp}^h(t)$ in  \eqref{errorIdentity} yields  
\begin{align*}
\begin{split}
\dot{\mathcal{E}}\big({\be}_{\kp, h}\big)  
+ \frac{1}{2}\dfrac{\text{d}}{\text{d}t}\big(\texttt{a} h_\cF^{-1}\jump{\jmath_\omega^+ \be_{\kp}^h}, \jump{\jmath_\omega^+ \be_{\kp}^h} \big)_{\cF^0_h}\leq  
 \dfrac{\text{d}}{\text{d}t} \big(\mean{\tfrac{1}{\rho} \bdiv_h \jmath_\omega^+ \be_{\kp}^h}, \jump{\jmath_\omega^+ \be_{\kp}^h}\big)_{\cF^0_h}  +  F(\dot{\be}_{\kp}^h),
\end{split}
\end{align*}
where we took into account that the term $A(\pi_2\dot{\be}_{\kp}^h, \jmath_\omega\dot{\be}_{\kp}^h) = (\omega \cV\pi_2\dot{\be}_{\kp, h},  \pi_2\dot{\be}_{\kp, h}  )$ is non-negative. Integrating the last estimate with respect to time we get   
\begin{align}\label{in1}
\begin{split}
\mathcal{E}\big(\be_{\kp, h}\big)  
+ \frac{\texttt{a}}{2} \norm{h_\cF^{-1/2}\jump{\jmath_\omega^+ \be_{\kp}^h} }_{0,\cF^0_h}^2 \leq 
  \big(\mean{\tfrac{1}{\rho} \bdiv_h \jmath_\omega^+ \be_{\kp}^h}, \jump{\jmath_\omega^+ \be_{\kp}^h}\big)_{\cF^0_h}  +  \int_0^t F(\dot{\be}_{\kp}^h)\, \text{d}s.
\end{split}
\end{align}
We will now estimate the different terms of the right-hand side of \eqref{in1} by using repeatedly  the Cauchy-Schwarz inequality, \eqref{contA}, and \eqref{ellipA}, followed by the  well known inequality $a b \leq \frac{a^2}{8} +  2 b^2$. Thanks to \eqref{discTrace},  the first term can be bounded as follows: 
\begin{align}\label{babor0}
\begin{split}
	\big(\mean{\tfrac{1}{\rho} \bdiv_h \jmath_\omega^+ \be_{\kp}^h}, \jump{\jmath_\omega^+ \be_{\kp}^h}\big)_{\cF^0_h} &\leq \norm{h_\cF^{1/2}\mean{\tfrac{1}{\rho} \bdiv_h \jmath_\omega^+ \be_{\kp}^h} }_{0,\cF^0_h} \norm{h_\cF^{-1/2}\jump{\jmath_\omega^+ \be_{\kp}^h} }_{0,\cF^0_h}
	\\[1ex]
	&\leq \frac{C_{\text{tr}}}{\norm{\rho^{-1}}_{L^{\infty}(\Omega)}} \norm{h_\cF^{-1/2}\jump{\jmath_\omega^+ \be_{\kp}^h} }_{0,\cF^0_h} \norm{ \bdiv_h \jmath_\omega^+ \be_{\kp}^h }_{0,\Omega}
	\\[1ex]
	&\leq \frac{1}{4}\max_{[0, T]}\mathcal{E}\big(\be_{\kp, h}\big) +  \frac{2C_{\text{tr}}^2}{\norm{\rho^{-1}}^2_{L^{\infty}(\Omega)}}  \norm{h_\cF^{-1/2}\jump{\jmath_\omega^+ \be_{\kp}^h} }_{0,\cF^0_h}^2.
	\end{split}
\end{align}
Next, we consider the splitting $\int_0^t F(\dot{\be}_{\kp}^h)\, \text{d}s = (1) + (2) + (3)$ and estimate each term individually as shown below. For the first term we have that  
\begin{align}\label{babor1}
\begin{split}
	(1) &:= -\int_0^t A\big( \jmath_\omega \ddot{\mathcal I}_{\kp}^h + \pi_2\dot{\mathcal I}_{\kp}^h, \jmath_\omega\dot{\be}_{\kp}^h\big)\, \text{d}s 
	\\[1ex]&
	\leq \sqrt{M} \max_{[0, T]} A\big( \jmath_\omega\dot{\be}_{\kp}^h, \jmath_\omega\dot{\be}_{\kp}^h\big)^{1/2}  \int_0^T \Big(\norm{\jmath_\omega\ddot{\mathcal I}_{\kp}^h}_{\mathfrak{L}_V^2(\Omega)} + \norm{\dot{\mathcal I}_{\kp}^h}_{\mathfrak{L}_V^2(\Omega)}\Big)\text{d}t
	\\[1ex]
	& \leq \frac{1}{4}\max_{[0, T]}\mathcal{E}\big(\be_{\kp, h}\big) + 2 M \Big( \int_0^T \norm{\jmath_\omega\ddot{\mathcal I}_{\kp}^h}_{\mathfrak{L}_V^2(\Omega)} + \norm{\dot{\mathcal I}_{\kp}^h}_{\mathfrak{L}_V^2(\Omega)}\text{d}t\Big)^2.
\end{split}	
\end{align}
For the second one it holds 
\begin{align}\label{babor2}
\begin{split}
	(2) &:= - \int_0^t (\dot{\mathcal I}_{\br}^h , \jmath_\omega^+\dot{\be}_{\kp}^h) \, \text{d}s \leq \max_{[0, T]} \norm{\jmath_\omega^+\dot{\be}_{\kp}^h}_{\mathfrak{L}_V^2(\Omega)} \int_0^T \norm{ \dot{\mathcal I}_{\br}^h }_{0,\Omega} \text{d}t 
	\\[1ex]&
	\leq \sqrt 2 \max_{[0,T]}\norm{\jmath_\omega \dot{\be}_{\kp}^h}_{\mathfrak{L}_V^2(\Omega)} \int_0^T \norm{ \dot{\mathcal I}_{\br}^h }_{0,\Omega} \text{d}t
	\leq \sqrt{\frac{2}{\alpha}}\max_{[0,T]} A\big(\jmath_\omega \dot{\be}_{\kp}^h, \jmath_\omega \dot{\be}_{\kp}^h  \big)^{1/2} \int_0^T \norm{ \dot{\mathcal I}_{\br}^h }_{0,\Omega} \text{d}t
	\\[1ex]
	&\leq    \frac{1}{4}\max_{[0, T]}\mathcal{E}\big(\be_{\kp, h}\big) + \frac{4}{\alpha} \Big( \int_0^T \norm{\dot{\mathcal I}_{\br}^h}_{0,\Omega} \Big)^2.
	\end{split}
\end{align}
Finally, using an integration by parts gives    
\begin{align}\label{babor3}
\begin{split}
	(3) &:= \int_0^t (\mean{\tfrac{1}{\rho} \bdiv \jmath_\omega^+ {\mathcal I}_{\kp}^h}, \jump{\jmath_\omega^+\dot{\be}_{\kp}^h})_{\cF^0_h} \,\text{d}s 
	\\[1ex]
	&= - \int_0^t (\mean{\tfrac{1}{\rho} \bdiv \jmath_\omega^+ \dot{{\mathcal I}}_{\kp}^h}, \jump{\jmath_\omega^+\be_{\kp}^h})_{\cF^0_h} \,\text{d}s 
	+ (\mean{\tfrac{1}{\rho} \bdiv \jmath_\omega^+ {\mathcal I}_{\kp}^h}, \jump{\jmath_\omega^+\be_{\kp}^h})_{\cF^0_h} \, \text{d}s 
	\\[1ex]
	&\leq   (1+T)\max_{[0, T]} \norm{h_\cF^{-1/2}\jump{\jmath_\omega^+ \be_{\kp}^h} }_{0,\cF^0_h} 
	\norm{ h_\cF^{1/2} \mean{\tfrac{1}{\rho} \bdiv \jmath_\omega^+ {\mathcal I}_{\kp}^h}}_{W^{1,\infty}(\bL^2(\cF^0_h))}
	\\[1ex]
	&\leq 2 \max_{[0, T]} \norm{h_\cF^{-1/2}\jump{\jmath_\omega^+ \be_{\kp}^h} }_{0,\cF^0_h}^2 + \frac{(T+1)^2}{8} \norm{ h_\cF^{1/2} \mean{\tfrac{1}{\rho} \bdiv \jmath_\omega^+ {\mathcal I}_{\kp}^h}}^2_{W^{1,\infty}(\bL^2(\cF^0_h))}.  
\end{split}
\end{align}

Plugging \eqref{babor0},  \eqref{babor1}, \eqref{babor2}, and \eqref{babor3}  in \eqref{in1} and rearranging terms we deduce that, if $\texttt{a}\geq \texttt{a}_0:= \frac{4C_{\text{tr}}^2}{\norm{\rho^{-1}}^2_{L^{\infty}(\Omega)}} + \frac{9}{4}$, then   
\begin{align*}
\frac{1}{4} \max_{[0, T]}\mathcal{E}\big({\be}_{\kp, h}\big)  
+  \frac{1}{4} \max_{[0, T]} \norm{h_\cF^{-1/2}\jump{\jmath_\omega^+ \be_{\kp}^h} }_{0,\cF^0_h}^2 \leq C \Big(  
\norm{{\mathcal I}_{\kp}^h}_{W^{2,\infty}(\mathfrak L_V^2(\Omega))}^2 + \norm{{\mathcal I}_{\br}^h}_{W^{1,\infty}(\mathbb L^2(\Omega))}^2 
\\[1ex]
+  \norm{ h_\cF^{1/2} \mean{\tfrac{1}{\rho} \bdiv \jmath_\omega^+ {\mathcal I}_{\kp}^h}}^2_{W^{1,\infty}(\bL^2(\cF^0_h))} \Big),
\end{align*}
with $C>0$ depending only on $T$, $\alpha$, $M$ and $\norm{\omega}_{L^{\infty}(\Omega)}$. Finally, using \eqref{eq-extra-1} and 
\[
\norm{\jmath_\omega{\be}_{\kp}^h(t)}_{\mathfrak{L}_V^2(\Omega)} = \norm{\int_0^t \jmath_\omega\dot{\be}_{\kp}^h(s)\, \text{d}s}_{\mathfrak{L}_V^2(\Omega)}\leq T \max_{[0,T]} \norm{\jmath_\omega \dot{\be}_{\kp}^h(t)}_{\mathfrak{L}_V^2(\Omega)}, 
\]
we conclude that 
\begin{align}\label{meal}
\begin{split}
\max_{[0, T]}\| \jmath_\omega\dot{\be}_{\kp}^h(t)\|_{\mathfrak{L}_V^2(\Omega)} + \max_{[0, T]}\|\be_{\kp}^h(t)\|_{\mathfrak S(h)} 
&\lesssim \norm{{\mathcal I}_{\kp}^h}_{W^{2,\infty}(\mathfrak L_V^2(\Omega))} + \norm{{\mathcal I}_{\br}^h}_{W^{1,\infty}(\mathbb L^2(\Omega))}
\\[1ex]
&\quad   +  \norm{ h_\cF^{1/2} \mean{\tfrac{1}{\rho} \bdiv \jmath_\omega^+ {\mathcal I}_{\kp}^h}}_{W^{1,\infty}(\bL^2(\cF^0_h))}. 
\end{split}
\end{align}

To estimate the error $\be^h_{\br}$, we first notice that integrating once with respect to time in  \eqref{errorIdentity} we obtain 
\begin{align}\label{errorIdentityr}
\begin{split}
(\be_{\br}^h, \jmath_\omega^+\kq) &= -A\big( \jmath_\omega\dot{\be}_{\kp}^h +  \pi_2{\be}_{\kp}^h, \jmath_\omega\kq\big)      
- \int_0^t \big(\bdiv_h \jmath_\omega^+\be_{\kp}^h(t),  \bdiv_h\jmath_\omega^+\kq \big)_\rho 
- \int_0^t \big(\texttt{a} h_\cF^{-1}\jump{\jmath_\omega^+ \be_{\kp}^h}, \jump{\jmath_\omega^+ \kq} \big)_{\cF^0_h}
\\[1ex]
&\quad + \int_0^t\big(\mean{\tfrac{1}{\rho} \bdiv_h \jmath_\omega^+ \be_{\kp}^h}, \jump{\jmath_\omega^+ \kq}\big)_{\cF^0_h} + \big(\mean{\tfrac{1}{\rho} \bdiv_h \jmath_\omega^+ \kq}, \jump{\jmath_\omega^+ \be_{\kp}^h}\big)_{\cF^0_h} + \int_0^t F(\kq),
\end{split}
\end{align}
for all $\kq \in \mathfrak S^{DG}_h$. Therefore, as  $\mathfrak S_h \subset \mathfrak S^{DG}_h$ and  $\norm{\kq}_{\mathfrak S(h)} = \norm{\kq}_{\mathfrak S}$ for all $\kq \in \mathfrak S_h$,    we deduce from the inf-sup condition  \eqref{discInfSup} that 
\begin{align*}
	\sup_{\kq\in \mathfrak S^{DG}_h}
\dfrac{(\be_{\br}^h, \jmath_\omega^+\kq )}{\norm{\kq}_{\mathfrak S(h)}} \geq \sup_{\kq\in \mathfrak S_h}
\dfrac{(\be_{\br}^h, \jmath_\omega^+\kq )}{\norm{\kq}_{\mathfrak S(h)}} = \sup_{\kq\in \mathfrak S_h}
\dfrac{(\be_{\br}^h, \jmath_\omega^+\kq )}{\norm{\kq}_{\mathfrak S}} \geq \beta^*\norm{\be_{\br}^h }_{0,\Omega}.
\end{align*}
Substituting identity \eqref{errorIdentityr} in the foregoing estimate and employing the Cauchy-Schwarz inequality, \eqref{discTrace}, and \eqref{meal} we get  
\begin{align*}
\norm{\be_{\br}^h }_{0,\Omega} 
\lesssim \norm{{\mathcal I}_{\kp}^h}_{W^{2,\infty}(\mathfrak L_V^2(\Omega))}
 + \norm{{\mathcal I}_{\br}^h}_{W^{1,\infty}(\mathbb L^2(\Omega))} +  \norm{ h_\cF^{1/2} \mean{\tfrac{1}{\rho} \bdiv \jmath_\omega^+ {\mathcal I}_{\kp}^h}}_{W^{1,\infty}(\bL^2(\cF^0_h))}.
\end{align*}
Finally, combining the last estimate with \eqref{meal} and using the triangle inequality  we conclude that 
\begin{align}\label{fint}
\begin{split}
	\max_{[0, T]}\norm{\jmath_\omega(\dot\kp-\dot\kp_h)}_{\mathfrak{L}_V^2(\Omega)} &+ \max_{[0, T]}\norm{\kp-\kp_h}_{\mathfrak S(h)} + \max_{[0,T]}
	\norm{\br - \br_h}_{0,\Omega}\lesssim \norm{{\mathcal I}_{\kp}^h}_{W^{2,\infty}(\mathfrak{S})} 
\\[1ex] 
&+  \norm{ h_\cF^{1/2} \mean{\tfrac{1}{\rho} \bdiv \jmath_\omega^+ {\mathfrak I}_{\kp}^h}}_{W^{1,\infty}(\bL^2(\cF^0_h))} + \norm{{\mathcal I}_{\br}^h}_{W^{1,\infty}(\mathbb L^2(\Omega))},
\end{split}
\end{align}
and the result follows.
\end{proof}

\begin{corollary}
If, besides the hypotheses of Theorem~\ref{convEstim}, we assume that $\kp \in \cC^2(\prod _{j=1}^J[\mathbb{H}^{k}(\Omega_j)]^{2})$, $\bdiv \jmath_\omega^+\kp \in \cC^2(\prod _{j=1}^J\bH^{k}(\Omega_j))$ and $\br\in \cC^1(\prod _{j=1}^J\mathbb{H}^{k}(\Omega_j))$, then 
\begin{equation}\label{asympSD}
\max_{[0, T]}\norm{\jmath_\omega(\dot\kp-\dot\kp_h)}_{\mathfrak{L}_V^2(\Omega)} + \max_{[0, T]}\norm{\kp-\kp_h}_{\mathfrak S(h)} + \max_{[0,T]} \norm{\br - \br_h }_{0,\Omega}  \lesssim  h^k.
\end{equation}
\end{corollary}
\begin{proof}
We first notice that by virtue of \eqref{div0}
\[
\tfrac{1}{\rho} \bdiv \jmath_\omega^+ {\mathcal I}_{\kp}^h   = \tfrac{1}{\rho} \left(\bdiv \jmath_\omega^+(\kp) - U_h  \bdiv\jmath_\omega^+\kp\right),
\]
and hence 
\begin{align*}
  \norm{h_\cF^{1/2} \mean{\tfrac{1}{\rho} \bdiv \jmath_\omega^+ {\mathcal I}_{\kp}^h}}^2_{0,\cF^0_h}  
  \lesssim    
   \sum_{K\in \cT_h} \sum_{F\in \cF(K)} h_F\norm{ \bdiv \jmath_\omega^+\kp - U_K \bdiv \jmath_\omega^+\kp) }^2_{0,F},
\end{align*}
where $U_K:= U_h|_K$ is the $\bL^2(K)$-orthogonal projection  onto $[\cP_{k-1}(K)]^d$. Under the regularity hypotheses on $\bdiv \jmath_\omega^+\kp$, the trace theorem and standard scaling arguments give 
\[
 h_F^{1/2}\norm{ \bdiv \jmath_\omega^+\kp - U_K \bdiv \jmath_\omega^+\kp) }_{0,F}  
 \lesssim h_K^k \norm{\bdiv \jmath_\omega^+\kp}_{k,K}, \quad \forall F\in \cF(K),
\]
which implies that 
\begin{equation}\label{bo1}
	\norm{ h_\cF^{1/2} \mean{\tfrac{1}{\rho} \bdiv \jmath_\omega^+ {\mathfrak I}_{\kp}^h}}_{W^{1,\infty}(\bL^2(\cF^0_h))} \lesssim h^k \sum_{j=1}^J \max_{[0, T]}\norm{\bdiv \jmath_\omega^+\kp}_{W^{1,\infty}(\bH^k(\Omega_j))}.
\end{equation}
Moreover, it is straightforward that 
\begin{equation}\label{asymQ}
 \norm{\bs- Q_h \bs}_{0,\Omega} \lesssim  h^k \sum_{j=1}^J \norm{\bs}_{k,\Omega_j}
 \qquad\forall\bs\in \prod _{j=1}^J\mathbb{H}^{k}(\Omega_j)\cap\mathbb Q.
\end{equation}
Hence, as a consequence of \eqref{XihStab} and \eqref{interpEs}, it holds
	\begin{align}\label{bo2}
	\begin{split}
		&\norm{{\mathcal I}_{\kp}^h}_{\W^{2,\infty}(\mathfrak{S})} + \norm{{\mathcal I}_{\br}^h}_{\W^{1,\infty}(\mathbb L^2(\Omega))} \leq  \norm{\kp - \boldsymbol{\Pi}_h \kp}_{\W^{2,\infty}(\mathfrak{S})} + \norm{\br - Q_h\br}_{\W^{1,\infty}(\mathbb L^2(\Omega))}
		\\[1ex]
		&\qquad \lesssim 
	h^k \sum_{j=1}^J\left\{  \norm{\kp}_{W^{1,\infty}([\mathbb  H^k(\Omega_j)]^2)} + \norm{\bdiv \jmath_\omega^+ \kp}_{W^{1,\infty}(\bH^k(\Omega_j))} + \norm{\br}_{\W^{1,\infty}(\mathbb H^k(\Omega_j))}\right\}.
	\end{split}
	\end{align}
	
	The asymptotic error estimate \eqref{asympSD} is now a direct consequence of \eqref{errorE}, \eqref{bo1}, and \eqref{bo2}.
\end{proof}

\section{Full discretization schemes}

We notice that the condition $\jmath_\omega^+\kq_h \in \mathbb H(\bdiv, \Omega)$ that applies to functions $\kq_h\in \mathfrak S_h$  translates into the continuity of the normal component of $\jmath_\omega^+\kq_h$ across all the internal facets $F\in \cF_h^0$. This restriction renders difficult the construction of an explicit basis of  $\mathfrak S_h$. Fortunately, the CG scheme \eqref{varFormR1-varFormR2-hc} can still be  efficiently implemented by hybridization. Indeed, if we let $\Theta_h := \bigoplus_{F\in \cF_h^0} [\cP_k(F)]^d$, then $\mathfrak S_h$ can be alternatively defined by  
\[
\mathfrak S_h = \big\{\kq \in \mathfrak S_h^{DG};\ \big(\boldsymbol{\phi}, \jump{\jmath_\omega^+\kq} \big)_{\cF_h} = 0,\quad \forall \boldsymbol{\phi}\in \Theta_h \big\}.
\]
Hence, we can relax in \eqref{varFormR1-varFormR2-hc} the continuity constraint on $\jmath_\omega^+ \kp_h(t)$ at the internal faces of the triangulation by introducing a Lagrange multiplier represented by an auxiliary trace variable $\boldsymbol{\psi}_h(t)\in \Theta_h$.  This leads to a hybrid mixed version of \eqref{varFormR1-varFormR2-hc} in which we look for $\kp_h\in \cC^1(\mathfrak S_h^{DG})$, $\br_h\in \cC^0(\mathbb Q_h)$, and $\boldsymbol{\psi}_h\in \cC^0(\Theta_h)$ satisfying 
\begin{align}\label{hybrid}
\begin{split}
\dfrac{\text{d}}{\text{d}t}  \Big\{ A\big( \jmath_\omega\dot\kp_h + \pi_2\kp_h , \jmath_\omega\kq\big)  + (\br_h,\jmath_\omega^+\kq)  - \big(\boldsymbol{\psi}_h,  \jump{\jmath_\omega^+\kq_h} \big)_{\cF^0_h}\Big\}
+ \big( \bdiv_h\jmath_\omega^+\kp_h + \bF,  \bdiv_h \jmath_\omega^+ \kq \big)_\rho & =   (\ddot{\boldsymbol{g}}, \jmath_\omega^+\kq\bn)_{\cF_h^\partial}, 
\\[1ex]
 (\bs, \jmath_\omega^+\kp_h) - \big(\boldsymbol{\phi}, \jump{\jmath_\omega^+\kp_h} \big)_{\cF^0_h}  &=  0, 
 \end{split}
 \end{align}
 for all $\kq\in \mathfrak{S}_h$ and $(\bs, \boldsymbol{\phi} )\in \mathbb{Q}_h\times \Theta_h$, and  such that the following initial conditions are satisfied:
\begin{equation*}
\kp_h(0)= \varXi_h\kp_0,\quad  \dot\kp_h(0) = \varXi_h\kp_1, \quad \br_h(0) = \mathbf 0, \quad \text{and} \quad  \boldsymbol{\psi}_h(0) = \mathbf 0.
\end{equation*}

We point out that the derivation of problem \eqref{hybrid} results from the same procedure used in Section~\ref{section3} to obtain \eqref{var0}. Indeed, it follows from \eqref{const+} that 
\begin{equation}\label{const+DG}
A\big( \jmath_\omega\dot\kp + \pi_2\kp , \jmath_\omega\kq\big)  = (\beps(\ddot{\bu}), \jmath_\omega^+ \kq)  = \big(\nabla \ddot\bu - \dot\br, \jmath_\omega^+ \kq\big),
\quad \forall \kq = (\bet, \btau)\in \mathfrak S_h^{DG}
\end{equation}
Performing now an integration by parts on each element $K$ and using that $\ddot \bu = \rho^{-1}( \bF + \bdiv \jmath_\omega^+ \kp)$ we obtain 
\begin{align*}\label{const+DG}
A\big( \jmath_\omega\dot\kp + \pi_2\kp , \jmath_\omega\kq\big)  &= (\beps(\ddot{\bu}), \jmath_\omega^+ \kq)  = -\big( \ddot\bu, \bdiv_h \jmath_\omega^+ \kq \big) + \big( \ddot \bu, \jump{\jmath_\omega^+ \kq} \big)_{\cF_h }
- \big( \dot\br, \jmath_\omega^+ \kq\big)
\\[1ex]
& = - \big( \bdiv\jmath_\omega^+\kp + \bF,  \bdiv_h \jmath_\omega^+ \kq \big)_\rho  + \big( \ddot \bu, \jump{\jmath_\omega^+ \kq} \big)_{\cF^0_h } + (\ddot{\boldsymbol{g}}, \jmath_\omega^+\kq\bn)_{\cF_h^\partial}- \big( \dot\br, \jmath_\omega^+ \kq\big), 
\end{align*}
for all $ \kq = (\bet, \btau)\in \mathfrak S_h^{DG}$. This reveals that the discrete function $\boldsymbol{\psi}_h(t)$ approximates the traces of  $\dot \bu(t) - \bu_1$ on the internal faces of the mesh. It is important to realize that the component $(\kp_h(t), \br_h(t))$ of the  solution to the hybrid problem \eqref{hybrid} coincides with the solution of the non-hybridized version \eqref{varFormR1-varFormR2-hc} of the problem. Indeed, the second equation of \eqref{hybrid} implies that $\kp_h(t)\in \mathfrak S_h$ and testing the first one with $\kq_h\in \mathfrak S_h$ we recover back  \eqref{varFormR1-varFormR2-hc}. 

We aim now to propose numerical time integration methods for the semi-discrete problems \eqref{hybrid} and \eqref{varFormR1-varFormR2-h}. To describe the form that these fully discrete schemes take we need to introduce  few notations. Given $L\in \mathbb{N}$, we consider a uniform partition of the time interval $[0, T]$ with step size $\Delta t := T/L$. Then, for any continuous function $\phi:[0, T]\to \R$ and for each $k\in\{0,1,\ldots,L\}$, we denote $\phi^k := \phi(t_k)$, where $t_k := k\,\Delta t$. In addition, we adopt the same notation for vector/tensor valued functions and consider $t_{k+\frac{1}{2}}:= \frac{t_{k+1} + t_k}{2}$, $\phi^{k+\frac{1}{2}}:= \frac{\phi^{k+1} + \phi^k}{2}$, $\phi^{k-\frac{1}{2}}:= \frac{\phi^{k} + \phi^{k-1}}{2}$, and $\widehat \phi^k := \frac{\phi^{k+\frac{1}{2}} + \phi^{k-\frac{1}{2}}}{2}= \frac{\phi^{k+1} + 2\phi^k + \phi^{k-1}}{4}$. We also introduce the discrete time derivatives
\[
\partial_t \phi^k := \frac{\phi^{k+1} - \phi^k}{\Delta t}, \quad \bar \partial_t \phi^k :=  \frac{\phi^{k} - \phi^{k-1}}{\Delta t}\quad\text{and} \quad \partial^0_t \phi^k := \frac{\phi^{k+1} - \phi^{k-1}}{2\Delta t}\,,
\]
from which we notice that $\partial_t \bar \partial_t \phi^k  = \frac{\phi^{k+1} -2\phi^{k} + \phi^{k-1}}{\Delta t^2}$.

\paragraph{Time-stepping scheme for the semi-discrete CG problem \eqref{hybrid}.}

We point out that large variations in the material parameters may require restrictive CFL conditions when explicit time-stepping schemes are used. With the purpose of ensuring  robustness, we opt for an implicit time integration method for \eqref{hybrid}. With this regard, Newmark's family of methods is one of the most widely used algorithms in structural dynamics (cf. \cite{hughes}). We will carry out the temporal discretization of our semi-discretized problems by applying the variante of Newmark's methods called the trapezoidal rule (also known in engineering as the average acceleration method). It is a second-order accurate and unconditionally stable time-integration method. In our case it reads as follows: For each $k=1,\ldots,L-1$, we look for $\kp_h^{k+1}\in \mathfrak S_h^{DG}$, $\br^{k+1}_h\in \mathbb{Q}_h$, and $\boldsymbol{\psi}^{k+1}_h\in \Theta_h$ such that
\begin{align}\label{fullyDiscretePb1-Pb2c}
\begin{split}
A\big( \jmath_{\omega}\partial_t\bar \partial_t\kp^k_h + \pi_2 \partial^0_t\kp^k_h , \jmath_{\omega}\kq\big)  &+ (\partial^0_t\br^k_h,\jmath_{\omega}^+\kq) - \big(\partial^0_t\boldsymbol{\psi}_h^k,  \jump{\jmath_\omega^+\kq} \big)_{\cF^0_h}+ 
\big( \bdiv_h\jmath_{\omega}^+ \widehat{\kp}^k_h,  \bdiv_h \jmath_{\omega}^+\kq \big)_\rho
\\[1ex]
  &= - \big( \bF(t_k),  \bdiv_h \jmath_{\omega}^+\kq \big)_\rho + \big<\ddot{\boldsymbol{g}}, \jmath_\omega^+\kq\bn \big>_\Gamma, \quad \forall \kq\in \mathfrak{S}_h^{DG}
 \\[1ex]
 &(\bs, \jmath^+_{\omega}\kp^{k+1}_h) + (\boldsymbol{\phi}, \jump{\jmath_\omega^+ \kp_h^{k+1}})_{\cF^0_h} = 0, \quad \forall (\bs,\boldsymbol{\phi}) \in \mathbb{Q}_h\times \Theta_h,
\end{split}
\end{align}
with the initial conditions
\begin{equation}\label{intconhyb}
	\kp_h^0= \varXi_h\kp_0,\quad \br_h^0 = \mathbf 0,\quad \boldsymbol{\psi}^0_h = \mathbf 0,
\quad
\text{and}
\quad
\kp_h^1= \varXi_h\kp_0 + \Delta t \, \varXi_h\kp_1 + \frac{\Delta t^2}{2} \varXi_h\ddot\kp(0)
\end{equation}
where the components of $\ddot\kp(0) = (\ddot\bgam(0),  \ddot\bze(0)) = \big( \cC\beps(\ddot\bu(0)),\,  \tilde\omega^{-1} \big( \cD \beps(\ddot\bu(0)) - \ddot\bgam(0) - \bze_1\big) \big)$ are deduced from
\[
\ddot\bu(0) = \rho^{-1}\big(\bF(0) + \bdiv\bsig_0 \big).
\] 

It is essential to notice that, the spaces $\mathfrak S_h^{DG}$ and $\mathbb Q_h$ have no interelement continuity requirements. Hence, at each $k=1,\ldots, L-1$, problem \eqref{fullyDiscretePb1-Pb2c} can be reduced by static condensation to a linear system that only involves the coefficients $\boldsymbol{\psi}^{k+1}_h$ as unknowns. The remaining variables $\kp_h^{k+1}$ and $\br_h^{k+1}$ can then be reconstructed after solving in $\boldsymbol{\psi}^{k+1}_h$ by performing computationally cheap element-wise operations.

\paragraph{Time-stepping scheme for the semi-discrete DG problem \eqref{varFormR1-varFormR2-h}.}

The DG method \eqref{varFormR1-varFormR2-h} is not hybridizable. Hence, an implicit time discretization method would not benefit in this case from a drastic size reduction due to static condensation, and it would generate a rather prohibitive computational cost. For this reason,  we propose here for \eqref{varFormR1-varFormR2-h} the following second order accurate explicit centered finite difference scheme: 

For each $k=1,\ldots,L-1$, we look for $\kp_h^{k+1}\in \mathfrak S_h^{DG}$ and $\br^{k+1}_h\in \mathbb{Q}_h$ such that
\begin{align}\label{fullyDiscretePb1-Pb2}
\begin{split}
A\big( \jmath_{\omega}\partial_t\bar \partial_t\kp^k_h &+ \pi_2 \partial^0_t\kp^k_h , \jmath_{\omega}\kq\big)  + (\partial^0_t\br^k_h,\jmath_{\omega}^+\kq) + 
\big( \bdiv_h\jmath_{\omega}^+ \kp^k_h,  \bdiv_h \jmath_{\omega}^+\kq \big)_\rho
\\[1ex]
&- \big(\mean{\tfrac{1}{\rho} \bdiv_h \jmath_\omega^+ \kp^k_h}, \jump{\jmath_\omega^+ \kq}\big)_{\cF^0_h}
- \big(\mean{\tfrac{1}{\rho} \bdiv_h \jmath_\omega^+ \kq}, \jump{\jmath_\omega^+ \kp^k_h}\big)_{\cF^0_h}+ \big(\texttt{a} h_\cF^{-1}\jump{\jmath_\omega^+ \kp^k_h}, \jump{\jmath_\omega^+ \kq} \big)_{\cF^0_h}
\\[1ex]
&
 = -\big(\bF(t_k),  \bdiv_h \jmath_{\omega}^+\kq \big)_\rho + \big(\mean{\tfrac{1}{\rho} \bF(t_k)}, \jump{\jmath_\omega^+ \kq}\big)_{\cF^0_h} + (\ddot{\boldsymbol{g}}, \jmath_\omega^+\kq\bn)_{\cF_h^\partial}, \quad \forall \kq\in \mathfrak{S}_h
 \\[1ex]
 & (\bs, \jmath^+_{\omega}\kp^{k+1}_h) = 0, \quad \forall \bs\in \mathbb{Q}_h,
\end{split}
\end{align}
where the initial values $\kp_h^0$, $\br_h^0$, and $\kp_h^1$ are given as in \eqref{intconhyb}. 

We point out that we can use here standard shape functions for $\cP_k(K)$ to expand the elements of $\mathfrak S_h^{DG}$. Actually, as the spaces $\mathfrak S_h^{DG}$ and $\mathbb Q_h$ are free from any interelement connexion, a judicious choice of locally orthogonal basis functions (see \cite{hesthaven}) renders the mass matrix of \eqref{fullyDiscretePb1-Pb2} diagonal and the corresponding time marching becomes then  fully explicit.

\end{document}